\documentclass[sort&compress, preprint,9pt]{elsarticle}

\usepackage[abs]{overpic}
\usepackage{graphicx}
\usepackage{subfigure}
\usepackage{stmaryrd}
\usepackage{amsfonts, color, amssymb}
\usepackage{amsmath}
\usepackage{amsthm}
\usepackage{epsfig}
\usepackage{psfrag}
\usepackage{bm}
\usepackage{paralist}
\usepackage{algorithm}
\usepackage{algorithmicx}
\usepackage{algpseudocode}
\usepackage{caption}
\usepackage{hyperref}
\usepackage[misc]{ifsym}
\usepackage[a4paper, total={6in,8in}]{geometry}

\allowdisplaybreaks[4]
\newtheorem{theorem}{Theorem}[section] 
\newtheorem{definition}{Definition}[section]

\newtheorem{example}{Example}[section]
\newtheorem{corollary}{Corollary}[section]

\newtheorem{lemma}{Lemma}[section]

\newcommand{\G}[1]{\nabla #1}
\newcommand{\GW}[1]{\nabla _w #1}
\newcommand{\p}[1]{\partial #1}
\newcommand{\normT}[1]{\norm{#1}_T}
\newcommand{\normPT}[1]{\norm{#1}_{\partial T}}
\newcommand{\pa}[2]{\frac{\partial #1}{\partial #2}}
\newcommand{\trb}[1]{|\!|\!|#1|\!|\!|}
\newcommand{\norm}[1]{\left\Vert#1\right\Vert}
\def\E{{\mathcal{E}}}
\def\T{{\mathcal{T}}}
\def\sumT{\sum_{T\in\mathcal{T}_h}}
\def\bq{{\mathbf{q}}}
\def\dQ{{\mathbb{Q}}}

\def\bn{{\bf n}}

\numberwithin{equation}{section}

\begin{document}

\begin{frontmatter}

\title{A Stabilizer-Free Weak Galerkin Mixed Finite Element Method for the Biharmonic Equation }

\author[mymainaddress]{Shanshan Gu}
\ead{guss22@mails.jlu.edu.cn}
\author[mymainaddress]{Fuchang Huo\corref{mycorrespondingauthor}}
\cortext[mycorrespondingauthor]{Corresponding author}
\ead{huofc22@mails.jlu.edu.cn}
\author[mymainaddress]{Shicheng Liu}
\ead{lsc22@mails.jlu.edu.cn}

\address[mymainaddress]{School of Mathematics, Jilin University, Changchun 130012, Jilin, China}

\begin{abstract}
In this paper, we present and research a stabilizer-free weak Galerkin (SFWG) finite element method for the Ciarlet-Raviart mixed form of the Biharmonic equation on general polygonal meshes.  We utilize the SFWG solutions of the second order elliptic problem to define projection operators and build error equations. Further, we derive the $O(h^k)$ and $O(h^{k+1})$ convergence for the exact solution $u$ in the $H^1$ and $L^2$ norms with polynomials of degree $k$. Finally, numerical examples support the results reached by the theory.
\end{abstract}


\end{frontmatter}

\section{Introduction}
\label{section:introduction}
Let $\Omega$ be a bounded polygonal domain in $\mathbb{R}^d$, $(d=2)$ with Lipschitz continuous boundary $\p \Omega$. We consider the Biharmonic equation as follows
    \begin{align}
    \label{Bmodel-equ1}\Delta^2u&=f, \qquad in~\Omega,\\
    \label{Bmodel-equ2}u&=0,\qquad on ~\partial \Omega,\\
    \label{Bmodel-equ3}\pa{u}{\bn}&=0,\qquad on ~\partial \Omega,
    \end{align}
    where $\bn$ is the unit outward normal vector on $\partial \Omega$.

    The variational form of $(\ref{Bmodel-equ1})-(\ref{Bmodel-equ3})$ is given as follows: find $u\in H^2_0(\Omega)$ such that
    \begin{align}
        \label{var-equ1}(\Delta u,\Delta v)=(f,v), \qquad \forall\,v\in H^2_0(\Omega),
    \end{align}
    where $H^2_0(\Omega)$ is defined by
    \begin{align*}
    H^2_0(\Omega)=\left\{v\in H^2(\Omega):~v|_{\partial\Omega}=0,~\pa{v}{\bn}|_{\partial\Omega}=0\right\}.
    \end{align*}

As a widely used technique, the conforming finite element methods have long been applied to the Biharmonic equation \cite{MR1191139, BiharmonicCFEM2, MR2817542}. They are based on the variational form (\ref{var-equ1}) and are required to construct finite element spaces as the subspaces of $H^2(\Omega)$, which need $C^1$-continuous finite elements. Due to the complexity of constructing high continuous elements, $H^2$-conforming finite element methods are seldom employed to solve the Biharmonic equation in actual computation.
    
To avoid constructing the $H^2$-conforming finite elements, nonconforming and discontinuous Galerkin finite element methods are also used to solve the equation. For example, the Morley element \cite{BiharmonicNCFEM} is a well-known nonconforming element for the Biharmonic equation, and a hp-version interior penalty discontinuous Galerkin method \cite{BiharmonicDGFEM} is proposed to solve the equation.
    
There are other ways that adopt different variational principles to deal with the problem. The hybrid finite element methods and mixed finite element methods do precisely that. For the analysis related to hybrid methods, see the paper of Brezzi \cite{MR391538} specifically. The mixed finite element methods build different variational forms from the above variational formulation by introducing an auxiliary variable. For example, the Ciarlet-Raviart form \cite{CRMFEM,BiharmonicMixFEM3} introduces a variable $\varphi = -\Delta u$ to build the variational formulation: find $u\in H^1_0(\Omega)$ and $\varphi\in H^1(\Omega)$ satisfying
\begin{align}
      \label{mixed-var-equ1}(\varphi ,v)-(\G u,\G v)&=0,\quad\quad\,\,\,\,\,\forall\, v\in H^1(\Omega),\\
      \label{mixed-var-equ2}(\G \varphi,\G \psi)&=(f,\psi),\quad\forall\, \psi\in H^1_0(\Omega).
\end{align}

This approach is more appropriate for hydrodynamics problems, where $-\Delta u$ represents vorticity. In addition, there are some other ways, such as the Hermann-Miyoshi method \cite{MR0386298} and Hermann-Johnson method \cite{HJMFEM}, which use the variable $\sigma = -\nabla ^2 u$ to build the variational form. The case is more appropriate for plate problems because of the second partial derivatives of the solution $u$, whose physical meanings are moments. For general results about mixed methods, the papers of Oden \cite{MFEMTheory} can be used as a reference.

The weak Galerkin (WG) finite element methods have been well developed as a new class of discontinuous Galerkin finite element methods in the last decade. 
The WG method employs weak differential operators to substitute classical differential operators in variational forms, facilitating the achievement of weak continuity in numerical solution through stabilizer.
 The method was first introduced in \cite{PossionWG} to solve the second order elliptic equation and achieved good results. The method is further developed by applying to more kinds of problems and modifying the definition of the weak differential operators. On the one hand, the WG method is utilized to solve the Stokes equation \cite{StokesWG}, the Biharmonic equation \cite{BiharmonicWG, BiharmonicWGReOrder}, the Brinkman equation \cite{BrinkmanWG}, etc. 
On the other hand, some scholars increase the degrees of the polynomial range space for the weak operator to eliminate the stabilizer in the WG numerical scheme \cite{PossionSFWG, StokesSFWG, BiharmonicSFWG}, which is known as the stabilizer-free weak Galerkin (SFWG) method. The SFWG method simplifies the numerical format and reduces the computation amount in the programming process.

Currently, the WG method \cite{BiharmonicWG,BiharmonicWGReOrder,SFC0WGBiharmonic} and SFWG method \cite{BiharmonicSFWG} are applied to the primal form of the Biharmonic equation. 
For the Ciarlet-Raviart mixed finite element formulation (\ref{mixed-var-equ1})-(\ref{mixed-var-equ2}), based on the Raviart-Thomas elements, a  WG numerical scheme without stabilizer for the mixed formulation is developed in \cite{BiharmonicWGMFEM}.
 However, the disadvantage of using the Raviart-Thomas elements is that the results only work on triangular meshes. In this paper, we use the SFWG method to discrete the Ciarlet-Raviart mixed formulation (\ref{mixed-var-equ1})-(\ref{mixed-var-equ2}) for the Biharmonic equation on general polytopal meshes. The convergence rates for the primal variable in the $H^1$ and $L^2$ norms are of order $O(h^k)$ and $O(h^{k+1})$, respectively.

The following is the outline of this paper. In Section \ref{Section:scheme}, we construct the SFWG numerical scheme for the Ciarlet-Raviart mixed formulation (\ref{mixed-var-equ1})-(\ref{mixed-var-equ2}). In Section \ref{Section:well-posedness}, we analyze the well-posedness of the SFWG method. In Section \ref{Section:error-analysis}, we present the error equations and give the $H^{1}$ and $L^2$ error estimates for the SFWG method. Numerical examples are shown in Section \ref{Section:Numex}. Section \ref{Section:conclusion} makes a conclusion.

\section{SFWG mixed scheme for the Biharmonic equation}\label{Section:scheme}

In this section, we shall introduce some simple notations and build the stabilizer-free numerical formulation.

Let $\T_h$ be the shape-regular partition of $\Omega$ satisfying the assumptions
in \cite{PossionMixedWG}.
Denote $\E_h$ as the set of all edges or flat faces in $\T_h$, and let $\E^0_h=\E_h\backslash\partial\Omega$ be the set of all interior edges or flat faces.
Let $h_T$ be the diameter of $T$, and denote the mesh size of $\T_{h}$ by $h=\max_{T\in\T_h}h_T$. Denote $\rho\in P_{k}(T)$
that $\rho|_{T}$ is polynomial with degree no more than $k$,
and the piecewise function space $P_{k}(e)$ is similar.

Let $K$ be an open bounded domain in $\mathbb{R} ^d$, and $s$ be a positive
integer. We use $\|\cdot\|_{s,K}$,
$|\cdot|_{s,K}$, $(\cdot,\cdot)_{s,K}$ to represent the norm, seminorm, and
inner product of the Sobolev space $H^s(K)$, respectively. We shall drop the subscript $K$ if $K=\Omega$ and drop the subscript $s$ if $s=2$. In addition, for any $a,b>0$, we use the term $a\lesssim b$ to express $a\leq Cb$, where $C>0$ is a constant independent of the mesh size.

For convenience, we use the following notations:
\begin{align*}
(v,w)_{\T_h}&=\sumT (v,w)_T=\sumT \int _T vwdT,\\
\langle v,w\rangle_{\partial \T_h}&=\sumT\langle v,w\rangle_{\partial T}=\sumT\int_{\partial T}vwds.
\end{align*}

Define the weak finite element spaces:
\begin{align}
  V_h&=\left\{v=\{v_0,v_b\},v_0|_T\in P_k(T),v_b|_e\in P_k(e),T\in \mathcal{T}_h,e\in \mathcal{E}_h\right\},\label{spa-Vh} \\
  V_h^0&=\{v=\left\{v_0,v_b\}\in V_h:~v_b|_{e}=0,~
  e\subset\partial\Omega\right\}.\label{spa-Vh0}
\end{align}

The discrete weak gradient of functions is defined as follows.
\begin{definition}\cite{PossionSFWG}
  For any $v\in V_h+H^1(\Omega)$,
  the discrete weak gradient $\GW v|_{T}\in [P_{j}(T)]^d$ satisfies
  \begin{align}\label{def-wgradient}
    (\GW v,\bq)_T = -(v_0,\G\cdot\bq)_T+\langle v_b,\bq\cdot\bn\rangle _{\p T},\qquad \forall\,\bq\in [P_j(T)]^d,
  \end{align}
  where $j$ will be defined later.
\end{definition}

In addition, we define
\begin{align*}
  a(w,v)=&~(w_0,v_0)_{\T _h}+\sumT h_T \langle w_0-w_b,v_0-v_b\rangle _{\p T},\quad\forall\, w,v\in V_h,\\
  b(v,\psi)=&~(\GW v,\GW \psi)_{\T _h},\quad\forall\, v\in V_h,\psi\in V^0_h.
\end{align*}
        
With these preparations, we can propose the SFWG numerical
scheme as follows.
\begin{algorithm}
\caption{Stabilizer-free Weak Galerkin Algorithm}
A stabilizer-free weak Galerkin numerical scheme of (\ref{mixed-var-equ1})-(\ref{mixed-var-equ2}) is seeking
$\varphi _h\times u_{h}\in V_h\times V^0_h$ such that
\begin{align}
  \label{scheme1}a(\varphi _h,v)-b(v,u_h)&=0,\quad\quad\,\,\,\,\,\,\,\,\forall\, v\in V_h,\\
  \label{scheme2}b(\varphi_h,\psi)&=(f,\psi _0),\quad\forall\,\psi\in V^0_h.
\end{align}
\end{algorithm}

By setting $v=\{1,1\}$ in (\ref{scheme1}), we can get an important fact that $\varphi_h$ has mean value zero over the domain $\Omega$. For simplicity, we define a space $\overline{V}_h\subset V_h$ as follows.
\begin{align*}
  \overline{V}_h=\left\{v=\{v_0,v_b\}\in V_h:\int_\Omega v_0dx=0\right\}.
\end{align*}

For the sake of later analysis, we introduce several $L^2$ projection operators. Let $Q_{0}$ be a locally defined $L^2$ projection operator to
$P_{k}(T)$ on each element $T\in\T_h$ and $Q_b$ be a locally defined $L^2$ projection operator to $P_k(e)$ on each edge $e\in\E_h$. Then $Q_h=\{Q_0,Q_b\}$ is the projection operator into $V_h$. Furthermore, we
define $\dQ_{h}$ as the projection to $[P_{j}(T)]^d$
in $T\in\T_h$.

\section{Well-posedness}\label{Section:well-posedness}
In this section, we first equip the space $V^0_h$ and $V_{h}$ with proper norms. Next, we use the introduced norms and the relationship between the norms to derive the well-posedness of the numerical scheme.

\begin{definition}\label{norm}
For any $v\in V_h+H^1(\Omega)$, we define the following semi-norms
\begin{align*}
\|v\|^2_{0,h}=&~a(v,v)=(v_0,v_0)_{\T _h}+\sumT h_T \langle v_0-v_b,v_0-v_b\rangle _{\p T},\\
\trb v ^2=&~b(v,v)=(\GW v,\GW v)_{\T_h},\\
\|v\|^2_{1,h}=&~\sumT \norm{\G v_0}^2_T+\sumT h_T^{-1}\norm{v_0-v_b}^2_{\partial T}.
\end{align*}
\end{definition}
Obviously, $\|\cdot\|_{0,h}$ is the norm in $V_h$. Besides, with respect to $\trb{\cdot}$ and $\norm{\cdot}_{1,h}$, there is the following relationship.

\begin{lemma}\label{lemma-norm-equ}
\cite{PossionSFWG}There exist two positive constants $C_1$ and $C_2$ such that
\begin{align}\label{norm-equ}
C_1\norm{v}_{1,h}\leq\trb v\leq C_2\norm{v}_{1,h},\quad \forall\, v\in V_h,
\end{align}
where $j=n+k-1$ ($n$ is the number of edges of the polygon) in the definition of $\GW$.
\end{lemma}
\begin{lemma}
  $\norm{\cdot}_{1,h}$ provides a norm in $V^0_h$ and $\overline{V}_h$.
\end{lemma}
\begin{proof}
  We shall only prove the positivity property for $\norm{\cdot}_{1,h}$. Assume that $\norm{v}_{1,h}=0$ for some $v\in V^0_h$. 
  From the definition of $\norm{\cdot}_{1,h}$, we have
  \begin{align*}
  \normT{\G v_0}=0, \quad \normPT{v_0-v_b}=0,\qquad \forall\, T\in \T_h.
  \end{align*}
  Due to $\normT{\G v_0}=0,~\forall\, T\in \T_h$, we get $v_0=const$ locally on each element T. Further, by using $\normPT{v_0-v_b}=0,~\forall\, T\in \T_h$, we obtain $v_0=v_b=const$ on $\Omega$. The boundary condition of $v_b=0$ finally implies $v=\{0,0\}$ on $\Omega$. Therefore, $\norm{\cdot}_{1,h}$ is the norm in $V_h^0$.

  On the other hand, assume that $\norm{v}_{1,h}=0$ for some $v\in \overline V_h$. From the definition of $\norm\cdot _{1,h}$, we get $v_0=v_b=const$ on $\Omega$, together with the fact that $v\in\overline V_h$ implies $v=\{0,0\}$ on $\Omega$. This completes the proof.
\end{proof}

Furthermore, the norm equivalence $(\ref{norm-equ})$ implies $\trb\cdot$ is the norm in $V_h^0$ and $\overline V _h$. And we have the following relationships about $\norm{\cdot}_{1,h}$, $\norm{\cdot}_{0,h}$ and $\trb{\cdot}$.

\begin{lemma}
  For any $v\in V_h$, we have
  \begin{align}
    \label{1h0h}\norm{v}_{1,h}\lesssim &~ h^{-1}\norm v_{0,h},\\
    \label{trb0h}\trb v\lesssim &~h^{-1}\norm v_{0,h}.
  \end{align}
\end{lemma}
\begin{proof}
  From the definition of $\norm\cdot _{1,h}$, the inverse inequality \cite[Lemma A.6]{PossionMixedWG}, and the definition of $\norm\cdot _{0,h}$, we obtain
  \begin{align*}
    \norm v_{1,h}^2=&~\sumT (\normT{\G v_0}^2+h^{-1}_T\normPT{v_0-v_b}^2)\\
    \lesssim&~\sumT (h^{-2}_T\normT{v_0}^2+h^{-1}_T\normPT{v_0-v_b}^2)\\
    \lesssim&~h^{-2}\sumT (\normT{v_0}^2+h_T\normPT{v_0-v_b}^2)\\
    \lesssim&~h^{-2}\norm{v}_{0,h}^2,
  \end{align*}
  which implies (\ref{1h0h}). Consequently, we have (\ref{trb0h}) by using Lemma $\ref{lemma-norm-equ}$.
\end{proof}

\begin{lemma}\label{lemma-0h-trb}
  We have the following estimates:
  \begin{align}
    \label{0htrb}\norm{v}_{0,h}\lesssim&~\trb v,\qquad\forall\, v\in V_h^0,\\
    \label{0htrbline}\norm{v}_{0,h}\lesssim&~\trb v,\qquad\forall\, v\in \overline V_h.
  \end{align}
\end{lemma}
\begin{proof}
  For any $v\in V_h^0$, it follows from \cite{BiharmonicWGMFEM} that there exists $\bq\in [H^1(\Omega)]^2$ such that $\G\cdot\bq=v_0$ and $\norm{\bq}_1\lesssim\norm{v_0}$. Then we obtain
  \begin{align*}
    \norm{v_0}^2 =&~ \sumT (v_0,\G\cdot\bq)_T\\
    =&~\sumT (-(\G v_0,\bq)_T+\langle v_0,\bq\cdot\bn\rangle _{\partial T})\\
    =&~\sumT (-(\G v_0,\mathbb{Q} _h\bq)_T+\langle v_0-v_b,\bq\cdot\bn\rangle _{\partial T})\\
    =&~\sumT ((v_0,\G\cdot\mathbb{Q} _h\bq)_T-\langle v_0,\mathbb{Q} _h\bq\cdot\bn \rangle_{\partial T} +\langle v_0-v_b,\bq\cdot\bn\rangle _{\partial T})\\
    =&~\sumT\left(-(\GW v,\mathbb Q_h\bq)_T+\langle v_b,\mathbb{Q} _h\bq\cdot\bn \rangle_{\partial T}-\langle v_0,\mathbb{Q} _h\bq\cdot\bn \rangle_{\partial T}+\langle v_0-v_b,\bq\cdot\bn\rangle _{\partial T}\right)\\
    =&~-(\GW v,\mathbb Q_h\bq)_{\T_h}+\sumT\langle v_0-v_b,(\bq-\mathbb Q_h\bq)\cdot\bn\rangle_{\partial T},
  \end{align*}
  where we have used the integration by parts, $\langle v_b,\bq\cdot\bn\rangle _{\p\T _h}=0$ and the definition of $\G _w$.\\
  By Cauchy-Schwarz inequality and $\norm{\bq}_1\lesssim\norm{v_0}$, we have
  \begin{align*}
    |(\GW v,\mathbb Q_h\bq)_{\T_h}|\lesssim&~\trb v\norm{\mathbb Q_h\bq}\\
    \lesssim&~\trb v\norm{\bq}\\
    \lesssim&~\trb v\norm{\bq}_1\\
    \lesssim&~\trb v\norm{v_0}.
  \end{align*}
  For $\sumT\langle v_0-v_b,(\bq-\mathbb Q_h\bq)\cdot\bn\rangle_{\partial T}$, we have
  \begin{align*}
    \sumT\langle v_0-v_b,(\bq-\mathbb Q_h\bq)\cdot\bn\rangle_{\partial T}\lesssim&~\left(\sumT h_T^{-1}\normPT{v_0-v_b}^2\right)^{\frac{1}{2}}\left(\sumT h_T\normPT{\bq-\mathbb Q_h\bq}\right)^{\frac{1}{2}}\\
    \lesssim&~h\norm{v}_{1,h}\norm{\bq}_1\\
    \lesssim&~h\trb v\norm{v_0},
  \end{align*}
  where we have used the Cauchy-Schwarz inequality, the trace inequality  \cite[Lemma A.3] {PossionMixedWG}, and the projection inequality \cite[Lemma 4.1]{PossionMixedWG}, and $\norm{\bq}_1\lesssim\norm{v_0}$.
  \\
  Thus, we get $\norm{v_0}^2 \lesssim \trb v\norm{v_0}$, which implies that
  \begin{align*}
    \norm{v_0}\lesssim \trb v.
  \end{align*}
\\
  According to the proof of Lemma \ref{lemma-norm-equ}, we have
  \begin{align}
    h\normPT{v_0-v_b}^2\lesssim&~h^{-1}\normPT{v_0-v_b}^2 \nonumber\\
    \lesssim&~\norm{\GW v}_T^2.
  \end{align}
  \\
  Combining the above results, we get (\ref{0htrb}).

  For (\ref{0htrbline}), if $v\in\overline{V}_h$, according to \cite{BiharmonicWGMFEM}, one may find a vector-valued function $\bq$ satisfying $\G\cdot \bq=v_0$ and $\bq\cdot\bn = 0$ on $\partial\Omega$. Apart from this, we have $\norm{\bq}_1\lesssim\norm{v_0}$. The remaining proof of the (\ref{0htrbline}) is similar to (\ref{0htrb}).

\end{proof}
  From Lemma \ref{lemma-norm-equ} and Lemma \ref{lemma-0h-trb}, we obtain
  \begin{align}
    \label{0h1h}\norm{v}_{0,h}\lesssim&~\norm{v}_{1,h},\qquad\forall\, v\in V_h^0,\\
    \label{0h1hline}\norm{v}_{0,h}\lesssim&~\norm{v}_{1,h},\qquad\forall\, v\in \overline V_h.
  \end{align}

  For all $\psi _h\in V_h^0$, we define
  \begin{align*}
    \trb{\psi _h}_1 = \left(\sumT\normT{Q_0(\G\cdot\GW\psi_h)}^2+\sum_{e\in\E_h} h^{-1}\norm{Q_b[\GW\psi_h]}^2_e\right)^\frac{1}{2},
  \end{align*}
  where $[\GW\psi_h]|_e=(\GW\psi_h)|_{\p T_1\cap e}\cdot\bn _1+(\GW\psi_h)|_{\p T_2\cap e}\cdot\bn _2$ for any internal edge $e\in \E_h^0$, $T_1$, $T_2$ are the elements sharing the edge $e$ and $\bn_1$, $\bn_2$ are the unit outward normal vectors of $T_1,T_2$ on $e$. When the edge $e$ is on the boundary $\partial\Omega$, which is the part of element boundary of $T$ and $\bn$ is the unit outward normal vector of $T$ on $e$, $[\GW\psi_h]|_e=(\GW\psi_h)|_{\p T\cap e}\cdot\bn$.

  Then, we have the following conclusions.

\begin{lemma}
  $\trb\cdot _1$ is a norm in $V_h^0$, and we have
  \begin{align}
    \label{bbound}b(v,\psi)\lesssim\norm{v}_{0,h}\trb\psi _1,\qquad\forall\, v\in V_h,\psi\in V_h^0,\\
    \label{infsup}\sup _{\forall\, v\in V_h} \frac{b(v,\psi)}{\norm{v}_{0,h}}\gtrsim \trb\psi _1\gtrsim \trb\psi,\qquad\forall\,\psi\in V^0_h.
  \end{align}
\end{lemma}
\begin{proof}
  We only prove $\psi=0$ if $\trb\psi _1=0$ to verify that $\trb\cdot _1$ is a norm. For any $\psi\in V^0_h$, if $\trb\psi _1=0$, we have
  \begin{align*}
    Q_0(\G\cdot\GW\psi_h)|_T=&~0,\quad\forall\, T\in \T_h,\\
    Q_b[\GW\psi_h]|_e=&~0,\quad\forall\, e\in\E_h.
  \end{align*}
  Further, from the definitions of $\GW$, $[\cdot]$, and $Q_h$, we have
  \begin{align*}
    (\GW \psi,\GW \psi)=&~-\sumT (\psi _0,\G\cdot\GW\psi)_T + \sum _{e\in\E_h}\langle \psi _b,[\GW \psi]\rangle _e\\
    =&~-\sumT (\psi _0,Q_0(\G\cdot\GW\psi))_T + \sum _{e\in\E_h}\langle \psi _b,Q_b[\GW \psi]\rangle _e\\
    =&~0,
  \end{align*}
  which shows $\trb\psi =0$. Since $\trb\cdot$ is a norm in $V^0_h$, we have $\psi =0$. Thus, $\trb\cdot _1$ is a norm in $V^0_h$.

  For (\ref{bbound}), by using the definition of $\GW$, the Cauchy-Schwarz inequality, the trace inequality, and the inverse inequality, we acquire
  \begin{align*}
    b(v,\psi)=&~\sumT (\GW v,\GW \psi)_T\\
    =&~\sumT ((-v_0,\G\cdot\GW\psi)_T+\langle v_b,\GW\psi\cdot\bn\rangle _{\partial T})\\
    =&~-\sumT (v_0,\G\cdot\GW\psi)_T + \sum_{e\in\E_h}\langle v_b,[\GW\psi] \rangle_e\\
    =&~-\sumT (v_0,Q_0(\G\cdot\GW\psi))_T + \sum_{e\in\E_h}\langle v_b,Q_b[\GW\psi]\rangle _e\\
    \leq&~ \left(\sumT \norm{v_0}_T^2\right)^{\frac{1}{2}}\left(\sumT\normT{Q_0(\G\cdot\GW\psi)}^2\right)^{\frac{1}{2}} \\
    &~+ \left(\sum_{e\in\E_h} h\norm{v_b}_e^2\right)^{\frac{1}{2}}\left(\sum_{e\in\E_h} h^{-1}\norm{Q_b[\GW\psi]}_e^2\right)^{\frac{1}{2}}\\
    \leq&~\left(\sumT (\norm{v_0}_T^2+h\norm{v_b}_{\partial T}^2)\right)^{\frac{1}{2}}\trb\psi _1\\
    \leq&~\left(\sumT (\norm{v_0}_T^2+h\norm{v_0-v_b}_{\partial T}^2+h\norm{v_0}_{\partial T}^2)\right)^{\frac{1}{2}}\trb\psi _1\\
    \lesssim&~ \norm{v}_{0,h}\trb\psi _1.
  \end{align*}
By using (\ref{0htrb}) gives
\begin{align*}
  \trb\psi ^2\lesssim&~ \norm{\psi}_{0,h}\trb\psi _1\\
  \lesssim&~\trb\psi\trb\psi _1,
\end{align*}
which yields
\begin{align}
  \label{trbtrb1}\trb\psi\lesssim\trb\psi _1,\qquad\forall\,\psi\in V_h^0.
\end{align}

  Next, we prove (\ref{infsup}). For any $\psi\in V^0_h$, we choose $v^\ast=\{-Q_0(\G\cdot\GW\psi),h^{-1}Q_b[\GW\psi]\}\in V_h$, and use the definitions of the $\GW$ and $Q_h$, which yields
  \begin{align*}
    b(v^\ast,\psi)=&~\sumT (\GW v^\ast,\GW \psi)_T\\
    =&~\sumT ((-v_0^\ast,\G\cdot\GW\psi)_T+\langle v_b^\ast,\GW\psi\cdot\bn\rangle _{\partial T})\\
    =&~-\sumT (v_0^\ast,\G\cdot\GW\psi)_T + \sum_{e\in\E_h}\langle v_b^\ast,[\GW\psi]\rangle_e\\
    =&~\sumT (Q_0(\G\cdot\GW\psi),\G\cdot\GW\psi)_T + \sum_{e\in\E_h}\langle h^{-1}Q_b[\GW\psi],[\GW\psi]\rangle _e\\
    =&~\sumT (Q_0(\G\cdot\GW\psi),Q_0(\G\cdot\GW\psi))_T + \sum_{e\in\E_h}\langle h^{-1}Q_b[\GW\psi],Q_b[\GW\psi]\rangle _e\\
    =&~\sumT \norm{Q_0(\G\cdot\GW\psi)}^2_T + \sum_{e\in\E_h}h^{-1}\norm{Q_b[\GW\psi]}_e^2\\
    =&~\trb\psi _1^2.
  \end{align*}
  Using the definitions of the $\norm{\cdot}_{0,h}$ and $v^\ast$, the trace inequality, and the inverse inequality, we have
  \begin{align*}
    \norm{v^\ast}^2_{0,h} =&~ \sumT \norm{v_0^\ast}_T^2 +\sumT h\norm{v_0^\ast-v_b^\ast}_{\partial T}^2\\
    =&~ \sumT \norm{-Q_0(\G\cdot\GW\psi)}_T^2 +\sumT h\norm{-Q_0(\G\cdot\GW\psi)-h^{-1}Q_b[\GW\psi]}_{\partial T}^2\\
    \lesssim&~ \sumT \norm{-Q_0(\G\cdot\GW\psi)}_T^2 +\sumT h\norm{Q_0(\G\cdot\GW\psi)}_{\partial T}^2 \\
    &~+\sumT h\norm{h^{-1}Q_b[\GW\psi]}_{\partial T}^2\\
    \lesssim&~ \sumT \norm{Q_0(\G\cdot\GW\psi)}_T^2 +\sumT h^{-1}\norm{Q_b[\GW\psi]}_{\partial T}^2\\
    \lesssim&~\trb\psi _1^2.
  \end{align*}
  Thus, we arrive at
  \begin{align*}
    \sup _{\forall v\in V_h} \frac{b(v,\psi)}{\norm{v}_{0,h}}\gtrsim \frac{b(v^\ast,\psi)}{\norm{v^\ast}_{0,h}}\gtrsim \trb\psi _1\gtrsim \trb\psi,
  \end{align*}
  which implies (\ref{infsup}).
\end{proof}

\begin{theorem}
  The numerical scheme (\ref{scheme1})-(\ref{scheme2}) exists a unique solution.
\end{theorem}
\begin{proof}
  Let $(\varphi _h^1, u_h^1)$ and $(\varphi _h^2, u_h^2)$ are two solutions of the numerical scheme (\ref{scheme1})-(\ref{scheme2}), then we have
  \begin{align}
    \label{scheme11}a(\varphi _h^1-\varphi _h^2,v)-b(v,u_h^1-u_h^2)=&~0,\qquad\forall\, v\in V_h,\\
    \label{scheme22}b(\varphi _h^1-\varphi _h^2,\psi)=&~0,\qquad\forall\, \psi\in V_h^0.
  \end{align}
  Let $v=\varphi _h^1-\varphi _h^2, \psi=u_h^1-u_h^2$ the difference between the two solutions, and add the two equations together, then we obtain
  \begin{align*}
    \norm{\varphi _h^1-\varphi _h^2}_{0,h}^2=0,
  \end{align*}
  which implies $\varphi _h^1=\varphi _h^2$. Furthermore, by (\ref{scheme11}), we have
  \begin{align*}
    b(v,u_h^1-u_h^2)=0,\qquad\forall v\in V_h.
  \end{align*}
By taking $v=u_h^1-u_h^2$, we get $\trb{u_h^1-u_h^2}=0$, which implies $u_h^1=u_h^2$. The proof is completed.
\end{proof}

\begin{lemma}
  For any $v\in H^{m+1}(\Omega)$, $(0\leq m\leq k)$, there holds
  \begin{align*}
    \norm{v-Q_hv}_{0,h}\lesssim h^m \norm{v}_m.
  \end{align*}
\end{lemma}
\begin{proof}
  Using the definition of $Q_h$, the trace inequality, and the projection inequality, we have
  \begin{align*}
    \norm{v-Q_hv}_{0,h}^2 =&~ \sumT \normT{v-Q_0v}^2+\sumT h\norm{v-Q_0v-(v-Q_bv)}_{\partial T}^2\\
    \lesssim&~\sumT \normT{v-Q_0v}^2+\sumT h\norm{v-Q_0v}_{\partial T}^2\\
    \lesssim&~\sumT (\normT{v-Q_0v}^2+h^2\norm{\G(v-Q_0v)}_{T}^2)\\
    \lesssim&~ h^{2m}\norm{v}_m^2,
  \end{align*}
  which completes the proof.
\end{proof}

\begin{lemma}
  $\forall\, v\in H^1(\Omega)$, there holds
  \begin{align}
    \label{gradEX}\GW v=\dQ_{h}(\G v).
  \end{align}
\end{lemma}
\begin{proof}
  It follows from the definition of the weak gradient, the integration by parts, and the definition of $\dQ_h$ that
  \begin{align*}
    (\GW v,\bq)_T=&~-(v,\G\cdot\bq)_T+\langle v,\bq\cdot\bn\rangle _{\partial T}\\
    =&~(\G v,\bq)_T\\
    =&~(\dQ_{h}(\G v),\bq)_T,\quad \forall\,\bq\in [P_j(T)]^d.
  \end{align*}
  Let $\bq = \GW v-\dQ_{h}(\G v)$, and then we  get (\ref{gradEX}).
\end{proof}

\section{Error analysis}\label{Section:error-analysis}
In this section, we shall derive the error equations by the newly introduced projection operators and make further error analysis.

\subsection{Ritz and Neumann projections}

Now, we shall introduce two projection operators, the Ritz projection $\Pi_h^R$ and the Neumann projection $\Pi_h^N$, which apply the SFWG method to the second order elliptic problem with different boundary conditions.

\begin{definition}
    For any $v\in H^1_0(\Omega)$, we define the Ritz projection $\Pi _h^Rv=\{\Pi _0^Rv,\Pi _b^Rv\}\in V_h^0$ as the solution of the following problem:
\begin{align}
  \label{Rh}(\GW \Pi_h^Rv,\GW \psi)_{\T _h}=(-\Delta v,\psi _0)_{\T _h},\qquad\forall\,\psi\in V_h^0.
\end{align}
\end{definition}

It is known that $\Pi_h^Rv$ is the stabilizer-free weak Galerkin finite element solution \cite{PossionSFWG} of the Poisson equation with homogeneous Dirichlet boundary condition.

\begin{definition}
    For any $v\in \overline H^1(\Omega)$, we define the Neumann projection $\Pi_h^N:\overline{H}^1(\Omega)\to \overline V_h$ such that
\begin{align}
  \label{Nh}(\GW \Pi_h^Nv,\GW \psi)_{\T _h}=(-\Delta v,\psi _0)_{\T _h}+\langle \G v\cdot\bn,\psi _b\rangle _{\partial \Omega},\qquad\forall\,\psi\in V_h,
\end{align}
where $\overline H^1(\Omega)=\left\{v\in H^1(\Omega): \int _\Omega vdx=0\right\}$.
\end{definition}

Similarly, the $\Pi_h^Nv=\{\Pi_0^Nv,\Pi_b^Nv\}\in \overline V_h$ can be seen as the stabilizer-free WG finite element solution of the Poisson equation with inhomogeneous Neumann boundary condition.

The relevant conclusions of these two projection operators are presented in \ref{appendix:RhNh}.

  \subsection{Error equations}

  Next, we shall derive the error equations for the SFWG numerical scheme (\ref{scheme1})-(\ref{scheme2}). 
  \begin{theorem}
      Define $\varepsilon _u=\Pi_h^Ru-u_h$ and $\varepsilon _\varphi =\Pi_h^N\varphi -\varphi _h$, we have the error equations as follows:
  \begin{align}
    \label{err_equ1}a(\varepsilon _\varphi,\phi _h)-b(\phi _h,\varepsilon _u)=&~E(\varphi ,u,\phi _h),\quad \forall\,\phi _h\in V_h,\\
    \label{err_equ2}b(\varepsilon _\varphi,\psi _h)=&~0,\quad\qquad\qquad\,\, \forall\,\psi _h\in V_h^0,
  \end{align}
  where $$E(\varphi ,u,\phi _h)=a(\Pi_h^N\varphi -\varphi ,\phi _h)+(\GW \phi _h,\GW (u-\Pi_h^Ru))_{\T_h}+l(u,\phi _0).$$
  \end{theorem}
  \begin{proof}
      Testing $\varphi =-\Delta u$ with $\phi _h\in V_h$ and using (\ref{poisson-err}), (\ref{Bmodel-equ3}), we have
  \begin{align*}
    (\varphi,\phi _0)_{\T_h}=&~-(\Delta u,\phi _0)_{\T_h}\\
    =&~(\GW u,\GW \phi _h)_{\T_h}+\langle (\dQ _h\G u-\G u)\cdot\bn,\phi _0-\phi _b\rangle _{\partial \T _h},
  \end{align*}
  which leads to
  \begin{align*}
    0=-(\varphi ,\phi _h)_{\T_h}+(\GW u,\GW \phi _h)_{\T_h}+l(u,\phi _h),\qquad \forall\,\phi _h\in V_h.
  \end{align*}
  $l(\cdot,\cdot)$ is defined in Lemma \ref{RhNh-err-equ}.
  Adding $a(\Pi_h^N\varphi,\phi _h)-(\GW \phi _h,\GW \Pi_h^Ru)_{\T _h}$ to the above equation, we get
  \begin{align*}
    &~a(\Pi_h^N\varphi,\phi _h)-(\GW \phi _h,\GW \Pi_h^Ru)_{\T _h}\\=&~a(\Pi_h^N\varphi -\varphi ,\phi _h)+(\GW \phi _h,\GW (u-\Pi_h^Ru))_{\T_h}+l(u,\phi _h),\qquad \forall\,\phi _h\in V_h.
  \end{align*}
  Using $\psi _h\in V_h^0$ to test $-\Delta \varphi =f$, (\ref{poisson-err}), $\psi _b |_{\partial \Omega}=0$ and (\ref{Nh_err_equ}), we obtain
  \begin{align*}
    (f,\psi _0)_{\T_h}=&~(-\Delta\varphi ,\psi _0)_{\T_h}\\
    =&~(\GW \varphi,\GW \psi _h)_{\T_h}+\langle (\dQ _h\G \varphi-\G \varphi)\cdot\bn,\psi _0-\psi _b\rangle _{\partial \T _h}\\
    =&~(\GW \varphi,\GW \psi _h)_{\T_h}+l(\varphi ,\psi _h)\\
    =&~(\GW \Pi_h^N\varphi,\GW \psi _h)_{\T _h}.
  \end{align*}
  Thus, we have
  \begin{align}
    \label{NhRh_equ1}a(\Pi_h^N\varphi,\phi _h)-(\GW \phi _h,\GW \Pi_h^Ru)_{\T _h}=&~E(\varphi ,u,\phi _h),\quad \forall\,\phi _h\in V_h,\\
    \label{NhRh_equ2}(\GW \Pi_h^N\varphi,\GW \psi _h)_{\T _h}=&~(f,\psi _h)_{\T_h},\quad\quad \forall\,\psi _h\in V_h^0.
  \end{align}
  Together with (\ref{scheme1})-(\ref{scheme2}), we get (\ref{err_equ1}) and (\ref{err_equ2}).
  \end{proof}
  
  \begin{lemma}
    Assume that $\varphi\in H^{m+1}(\Omega)$ and $u\in H^{n+1}(\Omega)$, where $1\leq m\leq k,\ 1\leq n\leq k$, we have
    \begin{align}
      \label{E_err}|E(\varphi ,u,\phi _h)|\lesssim ~(h^{m+1}\norm{\varphi}_{m+1}+h^{n-1}\norm{u}_{n+1})\norm{\phi _h}_{0,h}.
    \end{align}
  \end{lemma}
  \begin{proof}
    Using the definition of $\norm{\cdot}_{0,h}$ and (\ref{Nh_0h}), we have
    \begin{align*}
      |a(\Pi_h^N\varphi -\varphi,\phi _h)|\leq&~\norm{\Pi_h^N\varphi -\varphi}_{0,h}\norm{\phi _h}_{0,h}\\
      \lesssim&~h^{m+1}\norm{\varphi}_{m+1}\norm{\phi _h}_{0,h}.
    \end{align*}
    Using the definition of $\trb\cdot$, (\ref{Rh_trb}) and (\ref{trb0h}), we have
    \begin{align*}
      |(\GW (u-\Pi_h^Ru),\GW \phi _h)_{\T_h}|\leq&~\trb {u-\Pi_h^Ru}\trb{\phi _h}\\
      \lesssim&~h^n\norm{u}_{n+1}\trb{\phi _h}\\
      \lesssim&~h^n\norm{u}_{n+1} h^{-1}\norm{\phi _h}_{0,h}\\
      \lesssim&~h^{n-1}\norm{u}_{n+1}\norm{\phi _h}_{0,h}.
    \end{align*}
    By (\ref{l-est}) and (\ref{trb0h}), we get
    \begin{align*}
      |l(u,\phi _h)|\lesssim&~h^n\norm{u}_{n+1}\trb{\phi _h}\\
      \lesssim&~h^n\norm{u}_{n+1}h^{-1}\norm{\phi _h}_{0,h}\\
      \lesssim&~h^{n-1}\norm{u}_{n+1}\norm{\phi _h}_{0,h}.
    \end{align*}
    Thus, we obtain (\ref{E_err}).
  \end{proof}

  \subsection{Error estimates}
  Now, we utilize the above error equations to estimate the errors we want.

  \begin{theorem}
    Assume $\varphi\in H^{m+1}(\Omega),~u\in H^{n+1}(\Omega)$, $1\leq m\leq k,\ 1\leq n\leq k$, and arrive at
    \begin{align}
      \label{phi_0h}\norm{\varepsilon _\varphi}_{0,h}\lesssim&~ h^{m+1}\norm{\varphi}_{m+1}+h^{n-1}\norm u _{n+1}.
    \end{align}
  \end{theorem}
  \begin{proof}
    We set $\phi _h=\Pi_h^N\varphi -\varphi _h\in V_h,~\psi _h=\Pi_h^Ru -u_h\in V_h^0$ in (\ref{err_equ1})-(\ref{err_equ2}), add the two equations together, use (\ref{E_err}) and get
    \begin{align*}
      \norm{\varepsilon _\varphi}_{0,h}^2=&~E(\varphi ,u,\varepsilon _\varphi)\\
      \lesssim&~(h^{m+1}\norm{\varphi}_{m+1}+h^{n-1}\norm{u}_{n+1})\norm{\varepsilon _\varphi}_{0,h},
    \end{align*}
    which implies the (\ref{phi_0h}).
  \end{proof}

  \begin{theorem}
    Assume $\varphi\in H^{m+1}(\Omega),~u\in H^{n+1}(\Omega)$, $1\leq m\leq k,\ 1\leq n\leq k$, we have
    \begin{align}
      \label{u_L2}\norm{\varepsilon _{u,0}}\lesssim h^{m+1}\norm{\varphi}_{m+1}+h^{n+1}\norm{u}_{n+1}.
    \end{align}
    Further, we get
    \begin{align}
      \label{u_trb}\trb{\varepsilon _u}\lesssim&~h^{m}\norm{\varphi}_{m+1}+h^{n}\norm u _{n+1},\\
      \label{u_0h}\norm{\varepsilon _u}_{0,h}\lesssim&~ h^{m+1}\norm{\varphi}_{m+1}+h^{n+1}\norm{u}_{n+1}.
    \end{align}
  \end{theorem}
  \begin{proof}
    To derive the estimate for the $L^2$ norm of $\varepsilon _u$, we use the standard duality argument. Define
    \begin{align*}
      \xi +\Delta \eta =&~0,\qquad\, in ~\Omega ,\\
      -\Delta\xi =&~ \varepsilon _{u,0},\quad in ~\Omega ,\\
      \eta =\pa{\eta}{\bn}=&~0,\qquad\, on~\partial \Omega.
    \end{align*}
    Since we assume that all internal angles of $\Omega$ are less than $126.283696\cdots ^\circ$, the solution of the problem has $H^4$ regularity \cite{BiharmonicWGMFEM}:
    \begin{align}
      \label{regu}\norm{\xi}_2+\norm{\eta}_4\lesssim\norm{\varepsilon _{u,0}}.
    \end{align}
    Further, the second order elliptic problems with either the Dirichlet boundary condition or the Neumann boundary condition have $H^2$ regularity.

    For the above dual problems, like (\ref{NhRh_equ1})-(\ref{NhRh_equ2}), we have 
    \begin{align*}
      a(\Pi_h^N\xi,\phi _h)-(\GW \phi _h,\GW \Pi_h^R\eta)_{\T _h}=&~E(\xi ,\eta,\phi _h),\qquad \forall\phi _h\in V_h,\\
    (\GW \Pi_h^N\xi,\GW \psi _h)_{\T _h}=&~(\varepsilon _{u,0},\psi _h)_{\T_h},\qquad \forall\psi _h\in V_h^0.
    \end{align*}
    Define $$\Lambda (\Pi_h^N\xi ,\Pi_h^R\eta ;\phi _h,\psi _h)=a(\Pi_h^N\xi,\phi _h)-(\GW \phi _h,\GW \Pi_h^R\eta)_{\T _h}-(\GW \Pi_h^N\xi,\GW \psi _h)_{\T _h},$$ we know $$\Lambda (\phi _h,\psi _h;\Pi_h^N\xi ,\Pi_h^R\eta)=a(\phi _h,\Pi_h^N\xi)-(\GW \Pi_h^N\xi,\GW \psi _h)_{\T _h}-(\GW \phi _h,\GW \Pi_h^R\eta)_{\T _h}.$$
     It is not hard to see that $\Lambda$ is a symmetric bilinear form. Thus, we have 
    \begin{align*}
      \norm{\varepsilon _{u,0}}^2=&~E(\xi ,\eta ,\varepsilon _\varphi)-\Lambda (\Pi_h^N\xi ,\Pi_h^R\eta ;\varepsilon _\varphi ,\varepsilon _u)\\
      =&~E(\xi ,\eta ,\varepsilon _\varphi)-\Lambda (\varepsilon _\varphi ,\varepsilon _u ;\Pi_h^N\xi ,\Pi_h^R\eta)\\
      =&~E(\xi ,\eta ,\varepsilon _\varphi)-E(\varphi ,u,\Pi_h^N\xi),
    \end{align*}
    where we have used (\ref{err_equ1})-(\ref{err_equ2}). Next, we estimate the two items respectively.
    \\
    For $E(\xi ,\eta ,\varepsilon _\varphi)$, by (\ref{E_err}), (\ref{phi_0h}) and (\ref{regu}), we get 
    \begin{align*}
      |E(\xi ,\eta ,\varepsilon _\varphi)|\lesssim&~(h^2\norm{\xi}_2+h^2\norm{\eta}_4)\norm{\varepsilon _\varphi}_{0,h}\\
      \lesssim&~h^2(\norm{\xi}_2+\norm{\eta}_4)(h^{m+1}\norm{\varphi}_{m+1}+h^{n-1}\norm u _{n+1})\\
      \lesssim&~h^2\norm{\varepsilon _{u,0}}(h^{m+1}\norm{\varphi}_{m+1}+h^{n-1}\norm u _{n+1})
    \end{align*}
\\
    As to $E(\varphi ,u,\Pi_h^N\xi)$, by definition we discuss it in three parts.
\\
    Using (\ref{Nh_0h}) and the definition of $\norm{\cdot}_{0,h}$, we have
    \begin{align*}
      |a(\Pi_h^N\varphi -\varphi ,\Pi_h^N\xi)|\leq&~\norm{\Pi_h^N\varphi -\varphi}_{0,h}\norm{\Pi_h^N\xi}_{0,h}\\
      \lesssim&~h^{m+1}\norm{\varphi}_{m+1}\norm{\Pi_h^N\xi}_{0,h}\\
      \lesssim&~h^{m+1}\norm{\varphi}_{m+1}(\norm{\Pi_h^N\xi -\xi}_{0,h}+\norm{\xi}_{0,h})\\
      \lesssim&~h^{m+1}\norm{\varphi}_{m+1}(h^2\norm{\xi}_2+\norm{\xi})\\
      \lesssim&~h^{m+1}\norm{\varphi}_{m+1}\norm{\xi}_2\\
      \lesssim&~h^{m+1}\norm{\varphi}_{m+1}\norm{\varepsilon _{u,0}}.
    \end{align*}
    We utilize the Cauchy-Schwarz inequality, (\ref{Nh_trb}), (\ref{Rh_trb}), (\ref{gradEX}), the definition of $\GW$, the dual problems, (\ref{Bmodel-equ2}) and (\ref{Rh_L2}) to get
    \begin{align*}
      &~(\GW \Pi_h^N\xi ,\GW (u-\Pi_h^Ru))_{\T_h}\\
      =&~(\GW (\Pi_h^N\xi -\xi ),\GW (u-\Pi_h^Ru))_{\T_h}+(\GW \xi ,\GW (u-\Pi_h^Ru))_{\T_h}\\
      \leq&~\trb{\Pi_h^N\xi -\xi}\trb{u-\Pi_h^Ru}+(\dQ _h\G \xi ,\GW (u-\Pi_h^Ru))_{\T_h}\\
      \lesssim&~h^n\norm{u}_{n+1}h\norm{\xi}_2+(\G \xi ,\GW (u-\Pi_h^Ru))_{\T_h}\\
      \lesssim&~h^{n+1}\norm{u}_{n+1}\norm{\xi}_2-(u-\Pi_0^Ru,\Delta\xi)_{\T _h}+\langle u-\Pi_b^Ru,\G\xi\cdot\bn \rangle _{\partial\T _h}\\
      \lesssim&~h^{n+1}\norm{u}_{n+1}\norm{\xi}_2+(u-\Pi_0^Ru,\varepsilon _{u,0})_{\T _h}\\
      \lesssim&~h^{n+1}\norm{u}_{n+1}\norm{\varepsilon _{u,0}}+\norm{u-\Pi_0^Ru}\norm{\varepsilon _{u,0}}\\
      \lesssim&~h^{n+1}\norm{u}_{n+1}\norm{\varepsilon _{u,0}}.
    \end{align*}
By the definition of $l(\cdot ,\cdot)$, the Cauchy-Schwarz inequality, the trace inequality, the projection inequality, and (\ref{Nh_1h}), we  get
    \begin{align*}
      |l(u,\Pi_h^N\xi)|=&~|\langle (\dQ_h\G u-\G u)\cdot\bn, \Pi_0^N\xi -\Pi_b^N\xi \rangle _{\partial \T_h}|\\
      =&~|\langle (\dQ_h\G u-\G u)\cdot\bn, \Pi_0^N\xi -\xi-(\Pi_b^N\xi -\xi) \rangle _{\partial \T_h}|\\
      \leq&~\Big(\sumT h\norm{\dQ_h \G u-\G u}_{\partial T}^2\Big)^\frac{1}{2}\Big(\sumT h^{-1}\norm{\Pi_0^N\xi -\xi-(\Pi_b^N\xi -\xi)}_{\partial T}^2\Big)^\frac{1}{2}\\
      \lesssim&~\Big(\sumT \norm{\dQ_h \G u-\G u}_{T}^2+h^2\norm{\G (\dQ_h \G u-\G u)}_T^2\Big)^\frac{1}{2}\norm{\Pi_h^N\xi -\xi}_{1,h}\\
      \lesssim&~ h^n\norm{u}_{n+1}h\norm{\xi}_2\\
      \lesssim&~ h^{n+1}\norm{u}_{n+1}\norm{\xi}_2\\
      \lesssim&~h^{n+1}\norm{u}_{n+1}\norm{\varepsilon _{u,0}}.
    \end{align*}
    Therefore, we obtain 
    \begin{align*}
      \norm{\varepsilon _{u,0}}^2\lesssim (h^{m+1}\norm{\varphi}_{m+1}+h^{n+1}\norm{u}_{n+1})\norm{\varepsilon _{u,0}},
    \end{align*}
    which implies (\ref{u_L2}).

    To verify (\ref{u_0h}), we use the definitions of $\norm{\cdot}_{0,h}$ and $\norm{\cdot}_{1,h}$, (\ref{u_L2}) and Lemma \ref{lemma-norm-equ} to get 
    \begin{align*}
      \norm{\varepsilon _u}_{0,h} =&~ \Big(\sumT (\norm{\varepsilon _{u,0}}_T^2 + h\norm{\varepsilon _{u,0}-\varepsilon _{u,b}}_{\partial T}^2 )\Big)^\frac{1}{2}\\
      \lesssim&~ \norm{\varepsilon _{u,0}}+\Big(h^2\sumT h^{-1}\norm{\varepsilon _{u,0}-\varepsilon _{u,b}}_{\partial T}^2\Big)^\frac{1}{2}\\
      \lesssim&~ h^{m+1}\norm{\varphi}_{m+1}+h^{n+1}\norm{u}_{n+1}+h\norm{\varepsilon _{u}}_{1,h}\\
      \lesssim&~ h^{m+1}\norm{\varphi}_{m+1}+h^{n+1}\norm{u}_{n+1}+h\trb{\varepsilon _{u}}
    \end{align*}
    By setting $\phi _h=\varepsilon _u$ in (\ref{err_equ1}), (\ref{E_err}), (\ref{phi_0h}), the above inequality and the Young's inequality, we have
  \begin{align*}
    \trb{\varepsilon _u}^2=&~a(\varepsilon _\varphi,\varepsilon _u)-E(\varphi ,u,\varepsilon _u)\\
    \lesssim&~\norm{\varepsilon _\varphi}_{0,h}\norm{\varepsilon _u}_{0,h}+(h^{m+1}\norm{\varphi}_{m+1}+h^{n-1}\norm{u}_{n+1})\norm{\varepsilon _u}_{0,h}\\
    \lesssim&~(h^{m+1}\norm{\varphi}_{m+1}+h^{n-1}\norm{u}_{n+1})\norm{\varepsilon _u}_{0,h}\\
    \lesssim&~(h^{m+1}\norm{\varphi}_{m+1}+h^{n-1}\norm{u}_{n+1})(h^{m+1}\norm{\varphi}_{m+1}+h^{n+1}\norm{u}_{n+1}+h\trb{\varepsilon _{u}})\\
    \lesssim&~h^{2(m+1)}\norm{\varphi}_{m+1}^2+h^{m+n}\norm{\varphi}_{m+1}\norm{u}_{n+1}+h^{m+n+2}\norm{\varphi}_{m+1}\norm{u}_{n+1}\\
    &~+h^{2n}\norm{u}_{n+1}^2+(h^{m+2}\norm{\varphi}_{m+1}+h^n\norm{u}_{n+1})\trb{\varepsilon _{u}}\\
    \leq&~C(h^{2(m+1)}\norm{\varphi}_{m+1}^2+h^{2n}\norm{u}_{n+1}^2+h^{m+n}\norm{\varphi}_{m+1}\norm{u}_{n+1})+\frac{1}{2}\trb{\varepsilon _{u}}^2,
  \end{align*}
  which implies (\ref{u_trb}) and (\ref{u_0h}).
  \end{proof}
  \begin{corollary}
      For $\varphi\in H^{m+1}(\Omega)$, $u\in H^{n+1}(\Omega)$, $1\leq m\leq k,\ 1\leq n\leq k$, we have the following estimates:
      \begin{align}
          \label{trb-Qhphi-phih}\trb{Q_h\varphi -\varphi _h}&\lesssim h^{m}\norm{\varphi}_{m+1}+h^{n-2}\norm{u}_{n+1},\\
          \label{trb-Qhu-uh}\trb{Q_h u-u_h}&\lesssim h^{m}\norm{\varphi}_{m+1}+h^n\norm{u}_{n+1},\\
          \label{L2-Qhphi-phih}\norm{Q_h\varphi-\varphi_h}&\lesssim h^{m+1}\norm{\varphi}_{m+1}+h^{n-1}\norm{u}_{n+1},\\
          \label{L2-Qhu-uh}\norm{Q_hu-u_h}&\lesssim h^{m+1}\norm{\varphi}_{m+1}+h^{n+1}\norm{u}_{n+1}.
      \end{align}
  \end{corollary}
  \begin{proof}
      From Lemma \ref{Qhv_err}, Lemma \ref{trb_err}, (\ref{trb0h}), (\ref{phi_0h}), and (\ref{u_trb}), we  get (\ref{trb-Qhphi-phih}) and (\ref{trb-Qhu-uh}).

      By (\ref{QhRh_0h}), (\ref{QhNh_0h}), (\ref{phi_0h}), and (\ref{u_L2}), we have (\ref{L2-Qhphi-phih}) and (\ref{L2-Qhu-uh}).
  \end{proof}

\section{Numerical Results}\label{Section:Numex}
This section conducts numerical experiments to illustrate the convergence rates of the SFWG finite element method proposed in this study. The error for the SFWG solution is measured using the following norms:
\begin{align*}
    \trb{Q_hv-v_h}^2&=\sumT\int _T |\GW (Q_hv-v_h)|^2 dT,\quad (A~discrete~ H^1~norm)\\
    \norm{Q_hv-v_h}^2&=\sumT\int _T |Q_0v-v_0|^2 dT,\qquad \quad\,\,(A~discrete~L^2~norm)
\end{align*}

In the following computations, we employ uniform triangular grids, rectangular grids, and polygonal grids, as shown in Figures \ref{tri-level}-\ref{poly-level}, respectively.

\begin{figure}[h!]
	\centering
        \includegraphics[width=0.3 \columnwidth,height=0.3\linewidth]{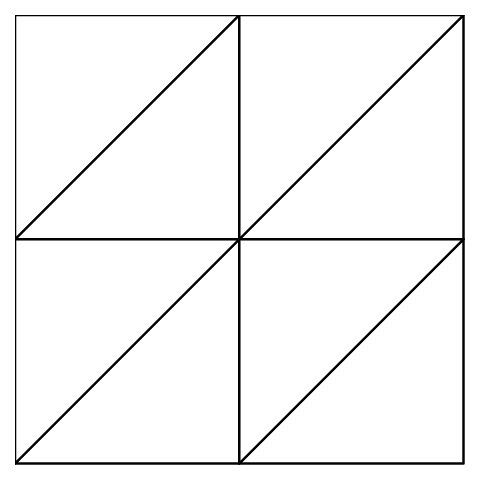}
        \includegraphics[width=0.3 \columnwidth,height=0.3\linewidth]{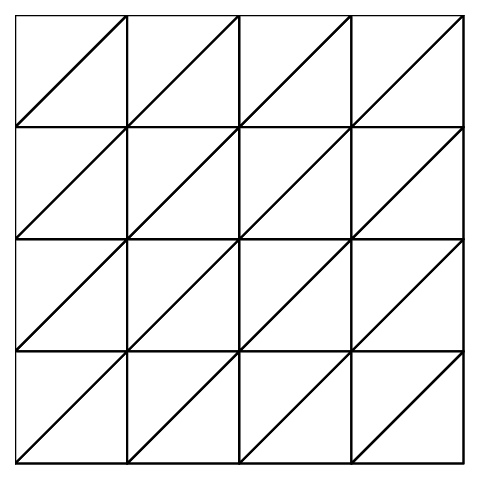}
        \includegraphics[width=0.3 \columnwidth,height=0.3\linewidth]{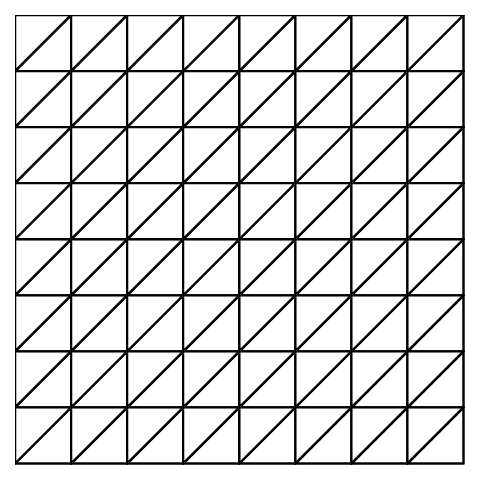}\\
	\vspace*{-4mm}
	\caption{The uniform triangular meshes  with $n=2,~4,~8$}
	\label{tri-level}
\end{figure}

\begin{figure}[h!]
	\centering
        \includegraphics[width=0.3 \columnwidth,height=0.3\linewidth]{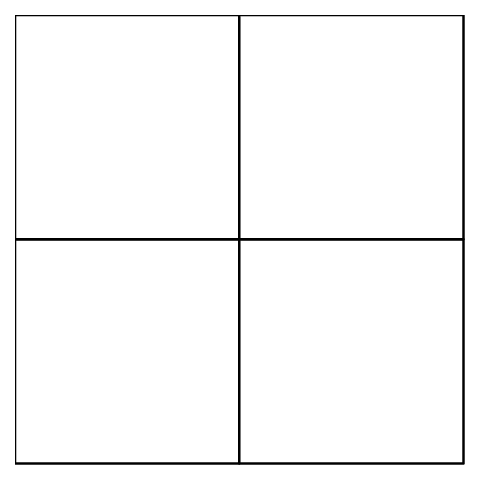}
        \includegraphics[width=0.3 \columnwidth,height=0.3\linewidth]{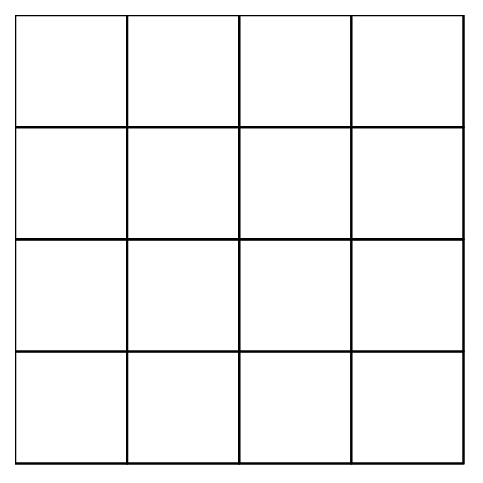}
        \includegraphics[width=0.3 \columnwidth,height=0.3\linewidth]{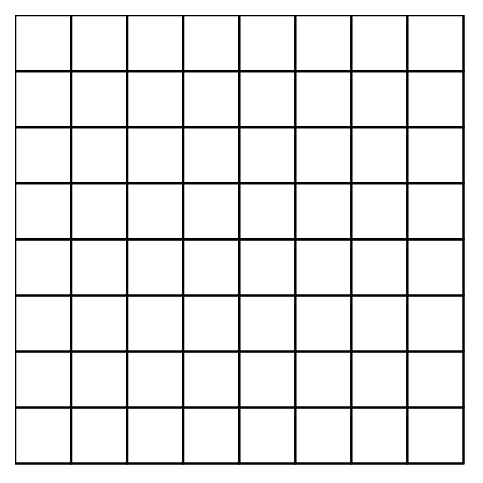}\\
	\vspace*{-4mm}
	\caption{The uniform rectangular meshes with $n=2,~4,~8$}
	\label{rect-level}
\end{figure}

\begin{figure}[h!]
	\centering
	\includegraphics[width=0.3          \columnwidth,height=0.3\linewidth]{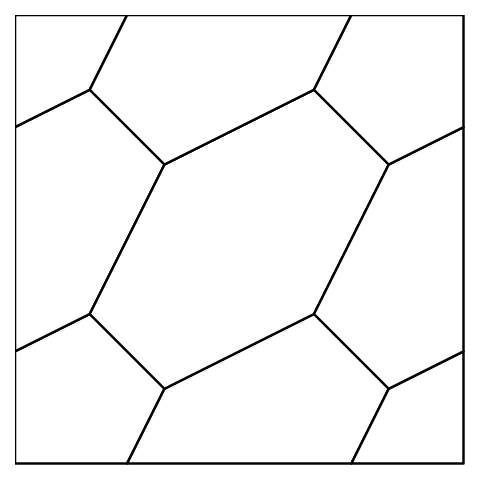}
        \includegraphics[width=0.3 \columnwidth,height=0.3\linewidth]{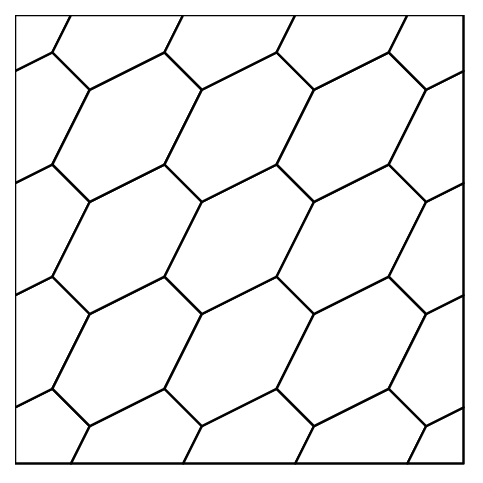}
        \includegraphics[width=0.3 \columnwidth,height=0.3\linewidth]{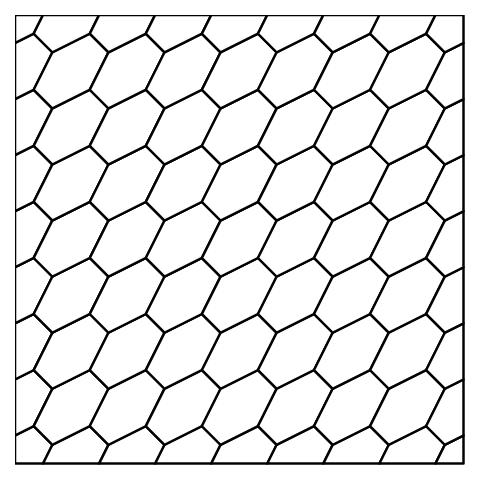}\\
	\vspace*{-4mm}
	\caption{The polygon meshes with $n=2,~4,~8$}
	\label{poly-level}
\end{figure}

\begin{example}\label{ex1} In this example, we consider the Biharmonic equation on the square domain $\Omega = (0,1)^2$, and 
the exact solution is chosen as follows
\begin{align*}
  u = x^2(1-x)^2y^2(1-y)^2.
\end{align*}
It is clear that $u|_{\partial\Omega}=0$ and $\frac{\partial u}{\partial\bn}|_{\partial\Omega}=0$.
\end{example}
The errors and convergence rates for the SFWG method are listed in  Tables \ref{ex1_tri}-\ref{ex1_poly}. From the tables, we can find the convergence rates for the error $Q_hu-u_h$ in the $H^1$ and $L^2$ norms are of order $O(h^{k})$ and $O(h^{k+1})$. The error convergence order is $O(h^{k-1})$ for $Q_h\varphi -\varphi_h$ in the $L^2$ norm. The numerical results are coincident with our theoretical analysis.

\begin{table}
  \begin{center}
    \caption{Errors and convergence rates on triangular meshes in Example \ref{ex1}}\label{ex1_tri}
    \resizebox{\textwidth}{!}{
\begin{tabular}{c|cc|cc|cc|cc}
  \hline $n$ & $\trb{Q_h\varphi -\varphi _h} $ & Rate & $\trb{Q_h u-u_h}$& Rate & $\norm{Q_h\varphi -\varphi _h}$ & Rate & $\norm{Q_hu-u_h}$ & Rate \\
  \hline & \multicolumn{8}{|c}{ By the $P_2$ weak Galerkin finite element } \\
    \hline 16 & $5.8675 \mathrm{E}-02$ & -- & $7.8153 \mathrm{E}-05$ & -- & $3.3439 \mathrm{E}-04$ & -- & $4.1789 \mathrm{E}-07$ & -- \\
     32 & $4.3407 \mathrm{E}-02$ & 0.43 & $1.9630 \mathrm{E}-05$ & 1.99 & $1.1637 \mathrm{E}-04$ & 1.52 & $5.0737 \mathrm{E}-08$ & 3.04 \\
     64 & $3.1476 \mathrm{E}-02$ & 0.46 & $4.9142 \mathrm{E}-06$ & 2.00 & $4.1265 \mathrm{E}-05$ & 1.50 & $6.2747 \mathrm{E}-09$ & 3.02 \\
     128 & $2.2556 \mathrm{E}-02$ & 0.48 & $1.2291 \mathrm{E}-06$ & 2.00 & $1.4663 \mathrm{E}-05$ & 1.49 & $7.8111 \mathrm{E}-10$ & 3.01 \\
  \hline & \multicolumn{8}{|c}{ By the $P_3$ weak Galerkin finite element } \\
    \hline 16 & $4.9264 \mathrm{E}-03$ & -- & $3.2862 \mathrm{E}-06$ & -- & $3.1500 \mathrm{E}-05$ & -- & $9.8363 \mathrm{E}-09$ & -- \\
     32 & $1.8548 \mathrm{E}-03$ & 1.41 & $4.1386 \mathrm{E}-07$ & 2.99 & $6.0192 \mathrm{E}-06$ & 2.39 & $5.9522 \mathrm{E}-10$ & 4.05 \\
     64 & $6.7498 \mathrm{E}-04$ & 1.46 & $5.1863 \mathrm{E}-08$ & 3.00 & $1.1038 \mathrm{E}-06$ & 2.45 & $3.6593 \mathrm{E}-11$ & 4.02 \\
     128 & $2.4189 \mathrm{E}-04$ & 1.48 & $6.4890 \mathrm{E}-09$ & 3.00 & $1.9855 \mathrm{E}-07$ & 2.47 & $2.2686 \mathrm{E}-12$ & 4.01 \\
  \hline
  \end{tabular}}
\end{center}
\end{table}

\begin{table}
  \begin{center}
    \caption{Errors and convergence rates on rectangular meshes in Example \ref{ex1}}\label{ex1_rect}
    \resizebox{\textwidth}{!}{
    \begin{tabular}{c|cc|cc|cc|cc}
    \hline $n$ & $\trb{Q_h\varphi -\varphi _h} $ & Rate & $\trb{Q_h u-u_h}$& Rate & $\norm{Q_h\varphi -\varphi _h}$ & Rate & $\norm{Q_hu-u_h}$ & Rate \\
    \hline & \multicolumn{8}{|c}{ By the $P_2$ weak Galerkin finite element } \\
    \hline 16 & $1.8588 \mathrm{E}-01$ & -- & $8.8362 \mathrm{E}-05$ & -- & $6.0219 \mathrm{E}-04$ & -- & $3.0180 E-06$ & -- \\
     32 & $1.4556 \mathrm{E}-01$ & 0.35 & $2.2756 \mathrm{E}-05$ & 1.96 & $2.1066 \mathrm{E}-04$ & 1.52 & $4.5616 E-07$ & 2.73 \\
     64 & $1.0805 \mathrm{E}-01$ & 0.43 & $5.7975 \mathrm{E}-06$ & 1.97 & $7.4677\mathrm{E}-05$ & 1.50 & $6.3231 E-08$ & 2.85 \\
     128 & $7.8249 \mathrm{E}-02$ & 0.47 & $1.4653 \mathrm{E}-06$ & 1.98 & $2.6628 \mathrm{E}-05$ & 1.49 & $8.3385 \mathrm{E}-09$ & 2.92 \\

    \hline & \multicolumn{8}{|c}{ By the $P_3$ weak Galerkin finite element } \\
    \hline 16 & $1.2964 \mathrm{E}-02$ & -- & $4.9116 \mathrm{E}-06$ & -- & $6.9365 \mathrm{E}-05$ & -- & $1.8319 \mathrm{E}-08$ & -- \\
     32 & $5.2127 \mathrm{E}-03$ & 1.31 & $6.3524 \mathrm{E}-07$ & 2.95 & $1.4350 \mathrm{E}-05$ & 2.27 & $9.7185 \mathrm{E}-10$ & 4.24 \\
     64 & $1.9622 \mathrm{E}-03$ & 1.41 & $8.0882 E-08$ & 2.97 & $2.7562 \mathrm{E}-06$ & 2.38 & $5.3668 \mathrm{E}-11$ & 4.18 \\
     128 & $7.1558 \mathrm{E}-04$ & 1.46 & $1.0208 \mathrm{E}-08$ & 2.99 & $5.0843 \mathrm{E}-07$ & 2.44 & $3.1624 \mathrm{E}-12$ & 4.08 \\
    \hline

    \end{tabular}}
  \end{center}
\end{table}

\begin{table}
  \begin{center}
    \caption{Errors and convergence rates on polygonal meshes in Example \ref{ex1}}\label{ex1_poly}
    \resizebox{\textwidth}{!}{
    \begin{tabular}{c|cc|cc|cc|cc}
    \hline $n$ & $\trb{Q_h\varphi -\varphi _h} $ & Rate & $\trb{Q_h u-u_h}$& Rate & $\norm{Q_h\varphi -\varphi _h}$ & Rate & $\norm{Q_hu-u_h}$ & Rate \\
    \hline & \multicolumn{8}{|c}{ By the $P_2$ weak Galerkin finite element } \\
  \hline 16 & $4.4235 \mathrm{E}-01$ & -- & $1.6183 \mathrm{E}-04$ & -- & $1.6034 \mathrm{E}-03$ & -- & $5.6366 \mathrm{E}-06$ & -- \\
   32 & $3.5657\mathrm{E}-01$ & 0.31 & $4.2222 \mathrm{E}-05$ & 1.94 & $6.1355 \mathrm{E}-04$ & 1.39 & $8.7874 \mathrm{E}-07$ & 2.68 \\
   64 & $2.6774 \mathrm{E}-01$ & 0.41 & $1.0802 \mathrm{E}-05$ & 1.97 & $2.2588 \mathrm{E}-04$ & 1.44 & $1.2324 \mathrm{E}-07$ & 2.83 \\
   128 & $1.9482 \mathrm{E}-01$ & 0.46 & $2.7343 \mathrm{E}-06$ & 1.98 & $8.1581 E-05$ & 1.47 & $1.6349 \mathrm{E}-08$ & 2.91 \\

    \hline & \multicolumn{8}{|c}{ By the $P_3$ weak Galerkin finite element } \\
    \hline 16 & $9.4109 \mathrm{E}-03$ & -- & $8.5082 \mathrm{E}-06$ & -- & $5.6358 \mathrm{E}-05$ & -- & $1.7179 \mathrm{E}-08$ & -- \\ 
     32 & $3.6245 \mathrm{E}-03$ & 1.38 & $1.1584 \mathrm{E}-06$ & 2.88 & $1.1518 \mathrm{E}-05$ & 2.29 & $1.1435 \mathrm{E}-09$ & 3.91 \\ 
     64 & $1.3389 \mathrm{E}-03$ & 1.44 & $1.5123 \mathrm{E}-07$ & 2.94 & $2.1957 \mathrm{E}-06$ & 2.39 & $7.4368 \mathrm{E}-11$ & 3.94 \\ 
     128 & $4.8379 \mathrm{E}-04$ & 1.47 & $1.9324 \mathrm{E}-08$ & 2.97 & $4.0316 \mathrm{E}-07$ & 2.45 & $4.7472 \mathrm{E}-12$ & 3.97\\
    \hline
    \end{tabular}}
  \end{center}
\end{table}

\begin{example}\label{ex2}
We choose the same solution area as in the above example. The exact solution is 
\begin{align*}
  u = \sin {(\pi x)}\sin{(\pi y)}.
\end{align*}
This shows that $u|_{\partial\Omega}=0$ and $\frac{\partial u}{\partial\bn}|_{\partial\Omega}\neq 0$.
\end{example}
The related results are shown in Tables \ref{ex2_tri}-\ref{ex2_poly}. From Tables \ref{ex2_tri}-\ref{ex2_poly}, we can see the same order of convergence as Example \ref{ex1}, which coincides with the theorem.
\begin{table}
  \begin{center}
    \caption{Errors and convergence rates on triangular meshes in Example \ref{ex2}}\label{ex2_tri}
    \resizebox{\textwidth}{!}{
    \begin{tabular}{c|cc|cc|cc|cc}
    \hline $n$ & $\trb{Q_h\varphi -\varphi _h} $ & Rate & $\trb{Q_h u-u_h}$& Rate & $\norm{Q_h\varphi -\varphi _h}$ & Rate & $\norm{Q_hu-u_h}$ & Rate \\
    \hline & \multicolumn{8}{|c}{ By the $P_2$ weak Galerkin finite element } \\
    \hline 16 & $9.1811 E+00$ & -- & $7.4385 \mathrm{E}-03$ & -- & $5.8928 \mathrm{E}-02$ & -- & $3.4903 \mathrm{E}-05$ & -- \\
     32 & $6.5266 \mathrm{E}+00$ & 0.49 & $1.8629 \mathrm{E}-03$ & 2.00 & $2.0861 \mathrm{E}-02$ & 1.50 & $4.1944 \mathrm{E}-06$ & 3.06 \\
     64 & $4.6215 \mathrm{E}+00$ & 0.50 & $4.6598 \mathrm{E}-04$ & 2.00 & $7.3777 \mathrm{E}-03$ & 1.50 & $5.1526 \mathrm{E}-07$ & 3.03 \\
     128 & $3.2691 E+00$ & 0.50 & $1.1652 \mathrm{E}-04$ & 2.00 & $2.6086 \mathrm{E}-03$ & 1.50 & $6.3906 \mathrm{E}-08$ & 3.01 \\

    \hline & \multicolumn{8}{|c}{ By the $P_3$ weak Galerkin finite element } \\
    \hline 16 & $2.5564 \mathrm{E}-01$ & -- & $2.0336 \mathrm{E}-04$ & -- & $1.6575 \mathrm{E}-03$ & -- & $5.8062 \mathrm{E}-07$ & -- \\
     32 & $9.0186 \mathrm{E}-02$ & 1.50 & $2.5500 \mathrm{E}-05$ & 3.00 & $2.9536 \mathrm{E}-04$ & 2.49 & $3.5466 \mathrm{E}-08$ & 4.03 \\
     64 & $3.1828 \mathrm{E}-02$ & 1.50 & $3.1915 \mathrm{E}-06$ & 3.00 & $5.2358 \mathrm{E}-05$ & 2.50 & $2.1929 \mathrm{E}-09$ & 4.02 \\
     128 & $1.1241 \mathrm{E}-02$ & 1.50 & $3.9914 \mathrm{E}-07$ & 3.00 & $9.2658 \mathrm{E}-06$ & 2.50 & $1.3635 \mathrm{E}-10$ & 4.01 \\
  \hline

    \end{tabular}}
  \end{center}
\end{table}

\begin{table}
  \begin{center}
    \caption{Errors and convergence rates on rectangular meshes in Example \ref{ex2}}\label{ex2_rect}
    \resizebox{\textwidth}{!}{
    \begin{tabular}{c|cc|cc|cc|cc}
    \hline $n$ & $\trb{Q_h\varphi -\varphi _h} $ & Rate & $\trb{Q_h u-u_h}$& Rate & $\norm{Q_h\varphi -\varphi _h}$ & Rate & $\norm{Q_hu-u_h}$ & Rate \\
    \hline & \multicolumn{8}{|c}{ By the $P_2$ weak Galerkin finite element } \\
    \hline 16 & $9.7153 \mathrm{E}+00$ & -- & $9.7509 \mathrm{E}-03$ & -- & $3.9732 \mathrm{E}-02$ & -- & $3.1204 \mathrm{E}-05$ & -- \\
     32 & $6.9398 \mathrm{E}+00$ & 0.49 & $2.4386 \mathrm{E}-03$ & 2.00 & $1.4154 \mathrm{E}-02$ & 1.49 & $2.7188 \mathrm{E}-06$ & 3.52 \\
     64 & $4.9198 \mathrm{E}+00$ & 0.50 & $6.0947 \mathrm{E}-04$ & 2.00 & $5.0143 \mathrm{E}-03$ & 1.50 & $2.8334 \mathrm{E}-07$ & 3.26 \\
     128 & $3.4811 \mathrm{E}+00$ & 0.50 & $1.5233 \mathrm{E}-04$ & 2.00 & $1.7738 \mathrm{E}-03$ & 1.50 & $3.3224 \mathrm{E}-08$ & 3.09 \\

    \hline & \multicolumn{8}{|c}{ By the $P_3$ weak Galerkin finite element } \\
    \hline 16 & $1.5239 \mathrm{E}-01$ & -- & $3.0553 \mathrm{E}-04$ & -- & $5.6483 \mathrm{E}-04$ & -- & $6.9886 \mathrm{E}-07$ & -- \\
     32 & $5.2705 \mathrm{E}-02$ & 1.53 & $3.8294 \mathrm{E}-05$ & 3.00 & $9.5049 \mathrm{E}-05$ & 2.57 & $4.3065 \mathrm{E}-08$ & 4.02 \\
     64 & $1.8493 \mathrm{E}-02$ & 1.51 & $4.7900 \mathrm{E}-06$ & 3.00 & $1.6481 \mathrm{E}-05$ & 2.53 & $2.6808 \mathrm{E}-09$ & 4.01 \\
     128 & $6.5198 \mathrm{E}-03$ & 1.50 & $5.9886 \mathrm{E}-07$ & 3.00 & $2.8899 \mathrm{E}-06$ & 2.51 & $1.6737 \mathrm{E}-10$ & 4.00 \\
    \hline
    \end{tabular}}
  \end{center}
\end{table}

\begin{table}
  \begin{center}
    \caption{Errors and convergence rates on polygonal meshes in Example \ref{ex2}}\label{ex2_poly}
    \resizebox{\textwidth}{!}{
    \begin{tabular}{c|cc|cc|cc|cc}
    \hline $n$ & $\trb{Q_h\varphi -\varphi _h} $ & Rate & $\trb{Q_h u-u_h}$& Rate & $\norm{Q_h\varphi -\varphi _h}$ & Rate & $\norm{Q_hu-u_h}$ & Rate \\
    \hline & \multicolumn{8}{|c}{ By the $P_2$ weak Galerkin finite element } \\
    \hline 16 & $5.1328 \mathrm{E}+01$ & -- & $1.2790 \mathrm{E}-02$ & -- & $2.4905 \mathrm{E}-01$ & -- & $2.1091 \mathrm{E}-04$ & -- \\
     32 & $3.6720 \mathrm{E}+01$ & 0.48 & $3.1524 \mathrm{E}-03$ & 2.02 & $8.8945 \mathrm{E}-02$ & 1.49 & $2.6967 \mathrm{E}-05$ & 2.97 \\
     64 & $2.6046 \mathrm{E}+01$ & 0.50 & $7.8318 \mathrm{E}-04$ & 2.01 & $3.1523 \mathrm{E}-02$ & 1.50 & $3.5157 \mathrm{E}-06$ & 2.94 \\
     128 & $1.8433 \mathrm{E}+01$ & 0.50 & $1.9526 \mathrm{E}-04$ & 2.00 & $1.1151 \mathrm{E}-02$ & 1.50 & $4.5238 \mathrm{E}-07$ & 2.96 \\
  \hline

    \hline & \multicolumn{8}{|c}{ By the $P_3$ weak Galerkin finite element } \\
    \hline 16 & $6.2345 \mathrm{E}-01$ & -- & $4.9235 \mathrm{E}-04$ & -- & $3.9976 \mathrm{E}-03$ & -- & $5.6491 \mathrm{E}-07$ & -- \\
     32 & $2.2038 \mathrm{E}-01$ & 1.50 & $6.3217 \mathrm{E}-05$ & 2.96 & $7.2734 \mathrm{E}-04$ & 2.46 & $3.1102 \mathrm{E}-08$ & 4.18 \\
     64 & $7.7966 \mathrm{E}-02$ & 1.50 & $8.0051 \mathrm{E}-06$ & 2.98 & $1.3035 \mathrm{E}-04$ & 2.48 & $1.8314 \mathrm{E}-09$ & 4.09 \\
     128 & $2.7581 \mathrm{E}-02$ & 1.50 & $1.0070 \mathrm{E}-06$ & 2.99 & $2.3197 \mathrm{E}-05$ & 2.49 & $1.1132 \mathrm{E}-10$ & 4.04 \\
    \hline

    \end{tabular}}
  \end{center}
\end{table}

\begin{example}\label{ex3}
We choose the same solution area as in the above examples. The exact solution is
\begin{align*}
  u = e^{x+y},
\end{align*}
which satisfies $u|_{\partial\Omega}\neq 0$ and $\frac{\partial u}{\partial\bn}|_{\partial\Omega}\neq 0$.
\end{example}
The relevant results are displayed in Tables \ref{ex3_tri}-\ref{ex3_poly}, which are consistent with the theoretical results.

\begin{table}
  \begin{center}
    \caption{Errors and convergence rates on triangular meshes in Example \ref{ex3}}\label{ex3_tri}
    \resizebox{\textwidth}{!}{
    \begin{tabular}{c|cc|cc|cc|cc}
    \hline $n$ & $\trb{Q_h\varphi -\varphi _h} $ & Rate & $\trb{Q_h u-u_h}$& Rate & $\norm{Q_h\varphi -\varphi _h}$ & Rate & $\norm{Q_hu-u_h}$ & Rate \\
    \hline & \multicolumn{8}{|c}{ By the $P_2$ weak Galerkin finite element } \\
    \hline 16 & $2.4698 \mathrm{E}+00$ & -- & $2.1195 \mathrm{E}-03$ & -- & $1.6056 \mathrm{E}-02$ & -- & $8.4842 \mathrm{E}-06$ & -- \\
     32 & $1.7491 \mathrm{E}+00$ & 0.50 & $5.3048 \mathrm{E}-04$ & 2.00 & $5.6647 \mathrm{E}-03$ & 1.50 & $1.0618 \mathrm{E}-06$ & 3.00 \\
     64 & $1.2376 \mathrm{E}+00$ & 0.50 & $1.3269 \mathrm{E}-04$ & 2.00 & $2.0004 \mathrm{E}-03$ & 1.50 & $1.3289 \mathrm{E}-07$ & 3.00 \\
     128 & $8.7543 \mathrm{E}-01$ & 0.50 & $3.3179 \mathrm{E}-05$ & 2.00 & $7.0679 \mathrm{E}-04$ & 1.50 & $1.6623 \mathrm{E}-08$ & 3.00 \\

    \hline & \multicolumn{8}{|c}{ By the $P_3$ weak Galerkin finite element } \\
    \hline 16 & $2.5285 \mathrm{E}-02$ & -- & $1.8753 \mathrm{E}-05$ & -- & $1.6657 \mathrm{E}-04$ & -- & $5.3828 \mathrm{E}-08$ & -- \\
     32 & $8.9885 \mathrm{E}-03$ & 1.49 & $2.3491 \mathrm{E}-06$ & 3.00 & $2.9758 \mathrm{E}-05$ & 2.48 & $3.2872 \mathrm{E}-09$ & 4.03 \\
     64 & $3.1862 \mathrm{E}-03$ & 1.50 & $2.9394 \mathrm{E}-07$ & 3.00 & $5.2877 \mathrm{E}-06$ & 2.49 & $2.0294 \mathrm{E}-10$ & 4.02 \\
     128 & $1.1326 \mathrm{E}-03$ & 1.49 & $3.6753 \mathrm{E}-08$ & 3.00 & $9.4181 \mathrm{E}-07$ & 2.49 & $1.2997 \mathrm{E}-11$ & 3.96 \\
    \hline

    \end{tabular}}
  \end{center}
\end{table}

\begin{table}
  \begin{center}
    \caption{Errors and convergence rates on polygonal meshes in Example \ref{ex3}}\label{ex3_poly}
    \resizebox{\textwidth}{!}{
    \begin{tabular}{c|cc|cc|cc|cc}
    \hline $n$ & $\trb{Q_h\varphi -\varphi _h} $ & Rate & $\trb{Q_h u-u_h}$& Rate & $\norm{Q_h\varphi -\varphi _h}$ & Rate & $\norm{Q_hu-u_h}$ & Rate \\
    \hline & \multicolumn{8}{|c}{ By the $P_2$ weak Galerkin finite element } \\
    \hline 16 & $1.2510 \mathrm{E}+01$ & -- & $3.3152 \mathrm{E}-03$ & -- & $6.3503 \mathrm{E}-02$ & -- & $2.6849 \mathrm{E}-05$ & -- \\
     32 & $9.1780 \mathrm{E}+00$ & 0.45 & $8.2523 \mathrm{E}-04$ & 2.01 & $2.3430 \mathrm{E}-02$ & 1.44 & $3.6017 \mathrm{E}-06$ & 2.90 \\
     64 & $6.6115 \mathrm{E}+00$ & 0.47 & $2.0561 \mathrm{E}-04$ & 2.00 & $8.4615 \mathrm{E}-03$ & 1.47 & $4.6959 \mathrm{E}-07$ & 2.94 \\
     128 & $4.7189 \mathrm{E}+00$ & 0.49 & $5.1299 \mathrm{E}-05$ & 2.00 & $3.0235 \mathrm{E}-03$ & 1.48 & $5.9940 \mathrm{E}-08$ & 2.97 \\

    \hline & \multicolumn{8}{|c}{ By the $P_3$ weak Galerkin finite element } \\
    \hline 16 & $5.4327 \mathrm{E}-02$ & -- & $4.3784 \mathrm{E}-05$ & -- & $3.5923 \mathrm{E}-04$ & -- & $4.3622 \mathrm{E}-08$ & -- \\
     32 & $1.9577 \mathrm{E}-02$ & 1.47 & $5.6822 \mathrm{E}-06$ & 2.95 & $6.5444 \mathrm{E}-05$ & 2.46 & $2.3973 \mathrm{E}-09$ & 4.19 \\
     64 & $6.9872 \mathrm{E}-03$ & 1.49 & $7.2357 \mathrm{E}-07$ & 2.97 & $1.1742 \mathrm{E}-05$ & 2.48 & $1.3815 \mathrm{E}-10$ & 4.12 \\
     128 & $2.4765 \mathrm{E}-03$ & 1.50 & $9.1265 \mathrm{E}-08$ & 2.99 & $2.0900 \mathrm{E}-06$ & 2.49 & $1.1196 \mathrm{E}-11$ & 3.63 \\
    \hline

    \end{tabular}}
  \end{center}
\end{table}

From Tables \ref{ex1_tri}-\ref{ex3_poly}, it can be seen that under various boundary conditions, the numerical examples consistently achieve the optimal convergence order, aligning with previous theoretical analyses and affirming the accuracy of the SFWG method. Notably, the convergence order of $\trb{Q_h\varphi -\varphi _h}$ is half order higher than the theoretical order $O(h^{k-2})$.

\section{Conclusion}\label{Section:conclusion}
In this paper, we propose a stabilizer-free weak Galerkin (SFWG) method for the Ciarlet-Raviart mixed variational form of the Biharmonic equation and analyze the well-posedness and convergence of the SFWG method.
We derive the optimal error estimates about the actual variable $u$ in the $H^1$ and $L^2$ norms. Finally, we use numerical examples to verify the theoretical analysis.
In future work, we will continue to study the WG related methods for the mixed form of the Biharmonic equation to reach the optimal error estimates in $H^1$ and $L^2$ norms for the auxiliary variable $\varphi$.

\section*{Acknowledgements}
This work was supported by the National Natural Science Foundation of China (grant No. 12271208, 12201246, 22341302), the National Key Research and Development Program of China (grant No. 2020YFA0713602, 2023YFA1008803), and the Key Laboratory of Symbolic Computation and Knowledge Engineering of Ministry of Education of China housed at Jilin University.

\section*{Data Availability}
The code used in this work will be made available
upon request to the authors.

\appendix
\section{Some Technical Results of Ritz and Neumann Projections}\label{appendix:RhNh}
\begin{lemma}\label{RhNh-err-equ}
For the definitions of $\Pi_h^R$ and $\Pi_h^N$, we have the following error equations:
    \begin{align}
  \label{Rh_err_equ}(\GW (\Pi_h^Rv-v),\GW \psi)_{\T _h}=&~l(v,\psi),\qquad\forall\,\psi\in V_h^0,\\
  \label{Nh_err_equ}(\GW (\Pi_h^Nv-v),\GW \psi)_{\T _h}=&~l(v,\psi),\qquad\forall\,\psi\in V_h,
\end{align}
where $$l(v,\psi)=\langle(\dQ _h\G v-\G v)\cdot\bn,\psi _0 -\psi _b\rangle_{\partial\T _h}.$$

For any $v\in H^{m+1}(\Omega),~(1\leq m\leq j+1)$ and for any $\psi\in V_h$, we can deduce the estimate as follows
\begin{align}
\big|l(v,\psi)\big|\lesssim~h^m\norm{v}_{m+1}\trb\psi.\label{l-est}
\end{align}
\end{lemma}
\begin{proof}
    Using the integration by parts, the definitions of $\dQ_h$ and $\GW$, (\ref{gradEX}), we have
\begin{align}
  (-\Delta v,\psi _0)_{\T _h}=&~ (\G v,\G \psi _0)_{\T _h}-\langle \G v\cdot\bn,\psi _0\rangle _{\partial\T _h}\nonumber \\
  =&~(\dQ _h\G v,\G\psi _0)_{\T _h}-\langle\G v\cdot\bn,\psi _0\rangle _{\partial\T _h}\nonumber \\
  =&~-(\G\cdot \dQ _h\G v,\psi _0)_{\T _h}+\langle (\dQ _h \G v-\G v)\cdot\bn,\psi _0\rangle _{\partial\T _h}\nonumber \\
  =&~(\dQ _h\G v,\GW \psi)-\langle\dQ _h\G v\cdot\bn,\psi _b\rangle _{\partial\T _h}+ \langle (\dQ _h \G v-\G v)\cdot\bn,\psi _0\rangle _{\partial\T _h}\nonumber \\
  =&~(\GW v,\GW\psi)_{\T _h}-\langle (\dQ _h\G v-\G v)\cdot\bn,\psi _b\rangle_{\partial\T _h}-\langle\G v\cdot\bn,\psi _b\rangle_{\partial\T _h}\nonumber \\
  &~+\langle (\dQ _h \G v-\G v)\cdot\bn,\psi _0\rangle _{\partial\T _h}\nonumber \\
  =&~(\GW v,\GW\psi)_{\T _h}+\langle (\dQ _h\G v-\G v)\cdot\bn,\psi _0 -\psi _b\rangle _{\partial\T _h}-\langle \G v\cdot\bn,\psi _b\rangle _{\partial\T _h}\label{poisson-err}.
\end{align}
By the above equation and the fact that $\psi |_{\partial\Omega}=0$ and (\ref{Rh}), we arrive at (\ref{Rh_err_equ}). Using (\ref{Nh}), we get (\ref{Nh_err_equ}). 

For (\ref{l-est}), we use the Cauchy-Schwarz inequality, the trace inequality, the projection inequality, the definition of $\norm{\cdot}_{1,h}$, (\ref{norm-equ}), and get
\begin{align*}
  \big|l(v,\psi)\big|=&~\Big |\langle(\dQ _h\G v-\G v)\cdot\bn,\psi _0 -\psi _b\rangle_{\partial\T _h}\Big |\\
  \lesssim&~\left( \sumT h_T\norm{\dQ _h\G v-\G v}_{\partial T}^2\right)^\frac{1}{2}\left(\sumT h_T^{-1}\norm{\psi _0-\psi _b}_{\partial T}^2\right)^{\frac{1}{2}}\\
  \lesssim&~h^m\norm{\G v}_m \norm{\psi}_{1,h}\\
  \lesssim&~h^m\norm{v}_{m+1}\trb\psi.
\end{align*}
Then we completes the proof.
\end{proof}

\begin{lemma}\label{Qhv_err}
  For any $v\in H^{m+1}(\Omega),(0\leq m\leq k)$, there holds
  \begin{align}
    \label{trbv_Qhv}\trb {v-Q_hv}\lesssim h^m\norm{v}_{m+1}.
  \end{align}
\end{lemma}
\begin{proof}
  We use the definition of the weak gradient, the integration by parts, the Cauchy-Schwarz inequality, the trace inequality, the inverse inequality, the definition of the projection operator, and the projection inequality to obtain for any $\bq \in [P_j(T)]^d$,
  \begin{align*}
    &~(\GW (v-Q_h v),\bq)_T \\
    &~= -(v-Q_0 v,\G\cdot\bq)_T+\langle v-Q_b v,\bq\cdot\bn\rangle _{\partial T}\\
    &~= (\G (v-Q_0 v),\bq)_T -\langle v-Q_0 v-(v-Q_b v),\bq\cdot\bn\rangle_{\partial T}\\
    &~\lesssim \norm{\G (v-Q_0 v)}\norm{\bq}+\left(\sumT h_T^{-1}\norm{Q_0v-Q_bv}_{\partial T}^2\right)^\frac{1}{2}\left(\sumT h_T\norm{\bq\cdot\bn}_{\partial T}^2\right)^\frac{1}{2}\\
    &~\lesssim \norm{\G (v-Q_0 v)}\norm{\bq}+\left(\sumT h_T^{-1}(\norm{v-Q_0v}_{\partial T}^2+\norm{v-Q_bv}_{\partial T}^2)\right)^\frac{1}{2}\norm{\bq}\\
    &~\lesssim \norm{\G (v-Q_0 v)}\norm{\bq}+\left(\sumT h_T^{-1}\norm{v-Q_0v}_{\partial T}^2\right)^\frac{1}{2}\norm{\bq}\\
    &~\lesssim \norm{\G (v-Q_0 v)}\norm{\bq}+\left(\sumT h_T^{-2}\norm{v-Q_0v}_{T}^2+\normT{\G (v-Q_0 v)}^2\right)^\frac{1}{2}\norm{\bq}\\
    &~\lesssim h^m\norm{v}_{m+1}\norm{\bq}.
  \end{align*}
  Let $\bq =\GW (v-Q_h v)$, then we get (\ref{trbv_Qhv}).
\end{proof}

\begin{lemma}\label{trb_err}
  For $v\in H^1_0(\Omega)\bigcap H^{m+1}(\Omega)$ or $\overline H ^1(\Omega)\bigcap H^{m+1}(\Omega)$, where $1\leq m\leq k$, we have
  \begin{align}
    \label{Rh_trb}\trb{v-\Pi_h^Rv}\lesssim&~ h^m\norm v _{m+1},\\
    \label{Nh_trb}\trb{v-\Pi_h^Nv}\lesssim&~ h^m\norm v _{m+1}.
  \end{align}
\end{lemma}
\begin{proof}
  For (\ref{Rh_trb}), by using (\ref{Rh_err_equ}), the estimate of $l(v,\psi)$ and the Young's inequality, we get
  \begin{align*}
    \trb{v-\Pi_h^Rv}^2=&~(\GW (v-\Pi_h^Rv),\GW (v-Q_hv))_{\T _h}+(\GW (v-\Pi_h^Rv),\GW (Q_hv-\Pi_h^Rv))_{\T _h}\\
    \lesssim&~\trb{v-Q_hv}\trb{v-\Pi_h^Rv}+\big|l(v,Q_hv-\Pi_h^Rv)\big|\\
    \lesssim&~\trb{v-Q_hv}\trb{v-\Pi_h^Rv}+h^m\norm v _{m+1}\trb{Q_hv-\Pi_h^Rv}\\
    \lesssim&~\trb{v-Q_hv}\trb{v-\Pi_h^Rv}+h^m\norm v _{m+1}(\trb{Q_hv-v}+\trb{v-\Pi_h^Rv})\\
    \leq&~ Ch^{2m}\norm v _{m+1}^2+\frac{1}{2}\trb{v-\Pi_h^Rv}^2.
  \end{align*}
  Thus, we have (\ref{Rh_trb}). Analogously, we use (\ref{Nh_err_equ}) and same inequalities to verify (\ref{Nh_trb}).
\end{proof}

  \begin{corollary}
      For $v\in H^1_0(\Omega)\bigcap H^{m+1}(\Omega)$ or $\overline H ^1(\Omega)\bigcap H^{m+1}(\Omega)$, where $1\leq m\leq k$, there hold
    \begin{align}
      \label{Rh_1h}\norm{v-\Pi_h^Rv}_{1,h}\lesssim&~ h^m\norm v _{m+1},\\
    \label{Nh_1h}\norm{v-\Pi_h^Nv}_{1,h}\lesssim&~ h^m\norm v _{m+1}.
    \end{align}
  \end{corollary}
  \begin{proof}
    Using Lemma \ref{lemma-norm-equ}, the definitions of $\norm{\cdot}_{1,h}$ and $Q_b$, the trace inequality, the projection inequality, Lemma \ref{Qhv_err} and Lemma \ref{trb_err}, we have
    \begin{align*}
      &~\norm{v-\Pi_h^Rv}_{1,h}\\
      \lesssim&~\norm{v-Q_hv}_{1,h}+\norm{Q_hv-\Pi_h^Rv}_{1,h}\\
      \lesssim&~\left(\sumT\normT{\G (v-Q_0v)}^2+h^{-1}\norm{v-Q_0v-(v-Q_bv)}_{\partial T}^2\right)^\frac{1}{2}+\trb{Q_hv-\Pi_h^Rv}\\
      \lesssim&~\left(\sumT\normT{\G (v-Q_0v)}^2+h^{-1}\norm{v-Q_0v-(v-Q_bv)}_{\partial T}^2\right)^\frac{1}{2}+\trb{Q_hv-v}+\trb{v-\Pi_h^Rv}\\
      \lesssim&~\left(\sumT\normT{\G (v-Q_0v)}^2+h^{-1}\norm{v-Q_0v}_{\partial T}^2\right)^\frac{1}{2}+h^m\norm v _{m+1}\\
      \lesssim&~h^m\norm v _{m+1},
    \end{align*}
    which completes the proof of (\ref{Rh_1h}). The (\ref{Nh_1h}) can be verified in the similar way.
  \end{proof}

  \begin{lemma}\label{L2_err}
    For $v\in H^1_0(\Omega)\bigcap H^{m+1}(\Omega)$ or $\overline H ^1(\Omega)\bigcap H^{m+1}(\Omega)$, where $1\leq m\leq k$, we have
    \begin{align}
      \label{QhRh_L2}\norm{Q_0v-\Pi_0^Rv}\lesssim&~ h^{m+1}\norm{v}_{m+1}\\
      \label{QhNh_L2}\norm{Q_0v-\Pi_0^Nv}\lesssim&~ h^{m+1}\norm{v}_{m+1}.
    \end{align}
    Further, we arrive at
    \begin{align}
      \label{Rh_L2}\norm{v-\Pi_0^Rv}\lesssim&~ h^{m+1}\norm{v}_{m+1}\\
      \label{Nh_L2}\norm{v-\Pi_0^Nv}\lesssim&~ h^{m+1}\norm{v}_{m+1}.
    \end{align}
  \end{lemma}
  \begin{proof}
    First, we use the standard duality argument to prove (\ref{QhNh_L2}) and (\ref{Nh_L2}) in detail. Find $\phi\in\overline H^1(\Omega)$ such that
    \begin{align}
      \label{dual-equ1}-\Delta \phi =&~ Q_0v-\Pi_0^Nv,\quad in ~\Omega,\\
      \label{dual-equ2}\pa{\phi}{\bn} =&~ 0,\quad \qquad \qquad\,\,\, on ~\partial \Omega,
    \end{align}
    and the above dual problem has the regularity assumption 
    $$\norm{\phi}_{2}\lesssim \norm{Q_0v-\Pi_0^Nv}.$$
    By (\ref{poisson-err}), (\ref{dual-equ2}), the Cauchy-Schwarz inequality, the projection inequality, (\ref{Nh_trb}), (\ref{Nh_err_equ}), and (\ref{l-est}), we get
    \begin{align*}
      &\norm{Q_0v-\Pi_0^Nv}^2\\
      =&~(-\Delta\phi,Q_0v-\Pi_0^Nv)_{\T_h}\\
      =&~(\GW\phi,\GW (Q_hv-\Pi_h^Nv))_{\T_h}+l(\phi,Q_hv-\Pi_h^Nv)\\
      =&~(\GW (\phi-\Pi_h^N\phi),\GW (Q_hv-\Pi_h^Nv))_{\T_h}+(\GW \Pi_h^N\phi,\GW (Q_hv-\Pi_h^Nv))_{\T_h}\\
      &~+l(\phi,Q_hv-\Pi_h^Nv)\\
      \lesssim&~ \trb{\phi - \Pi_h^N\phi}(\trb{Q_hv-v}+\trb{v-\Pi_h^Nv})+(\GW \Pi_h^N\phi,\GW (Q_hv-v))_{\T_h}\\
      &~+(\GW \Pi_h^N\phi,\GW (v-\Pi_h^Nv))_{\T_h}+l(\phi,Q_hv-\Pi_h^Nv)\\
      \lesssim&~h\norm{\phi}_{2}h^m\norm{v}_{m+1}+(\GW (\Pi_h^N\phi-\phi),\GW (Q_hv-v))_{\T_h}+(\GW \phi,\GW (Q_hv-v))_{\T_h}\\
      &~-l(v,\Pi_h^N\phi)+l(\phi,Q_hv-\Pi_h^Nv)\\
      \lesssim&~h^{m+1}\norm{v}_{m+1}\norm{\phi}_{2}+h\norm{\phi}_{2}h^m\norm{v}_{m+1}+(\GW \phi,\GW (Q_hv-v))_{\T_h}-l(v,\Pi_h^N\phi)\\
      &~+l(\phi,Q_hv-\Pi_h^Nv)\\
      \lesssim&~h^{m+1}\norm{v}_{m+1}\norm{\phi}_{2}+(\GW \phi,\GW (Q_hv-v))_{\T_h}-l(v,\Pi_h^N\phi)\\
      &~+h\norm{\phi}_{2}(\trb{Q_hv-v}+\trb{v-\Pi_h^Nv})\\
      \lesssim&~h^{m+1}\norm{v}_{m+1}\norm{\phi}_{2}+(\GW \phi,\GW (Q_hv-v))_{\T_h}-l(v,\Pi_h^N\phi).
    \end{align*}
    For $(\GW \phi,\GW (Q_hv-v))_{\T_h}$, by using (\ref{gradEX}), the definitions of the projection operator and the weak gradient $\GW$, (\ref{dual-equ1}), (\ref{dual-equ2}), the Cauchy-Schwarz inequality, and the projection inequality, we have
    \begin{align*}
      (\GW \phi,\GW (Q_hv-v))_{\T_h}=&~(\dQ_h\G \phi,\GW (Q_hv-v))_{\T_h}\\
      =&~(\G \phi,\GW (Q_hv-v))_{\T_h}\\
      =&~-(Q_0v-v,\G\cdot\G\phi)_{\T_h}+\langle Q_bv-v,\G\phi\cdot\bn\rangle_{\partial\T_h}\\
      =&~(Q_0v-v,Q_0v-\Pi_0^Nv)_{\T_h}\\
      =&~\norm{Q_0v-v}\norm{Q_0v-\Pi_0^Nv}\\
      \lesssim&~h^{m+1}\norm{v}_{m+1}\norm{Q_0v-\Pi_0^Nv}
    \end{align*}
    As to $l(v,\Pi_h^N\phi)$, we have
    \begin{align*}
      l(v,\Pi_h^N\phi)=&~\langle (\dQ_h\G v-\G v)\cdot\bn,\Pi_0^N\phi-\Pi_b^N\phi\rangle _{\partial\T_h}\\
      \leq&~\Big(\sumT h_T\norm{\dQ_h\G v-\G v}_{\partial T}^2\Big)^\frac{1}{2}\Big(\sumT h_T^{-1}\norm{\Pi_0^N\phi-\Pi_b^N\phi}_{\partial T}^2\Big)^\frac{1}{2}\\
      \lesssim&~\Big(\sumT \norm{\dQ_h\G v-\G v}_{T}^2+h_T^{2}\norm{\G(\dQ_h\G v-\G v)}_{T}^2\Big)^\frac{1}{2}\\
      &~\Big(\sumT h_T^{-1}\norm{\Pi_0^N\phi-\phi-(\Pi_b^N\phi-\phi)}_{\partial T}^2\Big)^\frac{1}{2}\\
      \lesssim&~h^m\norm v_{m+1}\norm{\Pi_h^N\phi-\phi}_{1,h}\\
      \lesssim&~h^{m+1}\norm{v}_{m+1}\norm{\phi}_{2},
    \end{align*}
    where we have utilized the Cauchy-Schwarz inequality, the trace inequality, the projection inequality and (\ref{Nh_1h}).
    Thus, we get
    \begin{align*}
      \norm{Q_0v-\Pi_0^Nv}^2\lesssim&~h^{m+1}\norm{v}_{m+1}\norm{\phi}_{2}+h^{m+1}\norm{v}_{m+1}\norm{Q_0v-\Pi_0^Nv}\\
      \lesssim&~h^{m+1}\norm{v}_{m+1}\norm{Q_0v-\Pi_0^Nv}+h^{m+1}\norm{v}_{m+1}\norm{Q_0v-\Pi_0^Nv}\\
      \lesssim&~h^{m+1}\norm{v}_{m+1}\norm{Q_0v-\Pi_0^Nv},
    \end{align*}
    which implies $$\norm{Q_0v-\Pi_0^Nv}\lesssim h^{m+1}\norm{v}_{m+1}.$$
    Furthermore, we have
    \begin{align*}
      \norm{v-\Pi_0^Nv}\leq&~\norm{v-Q_0v}+\norm{Q_0v-\Pi_0^Nv}\\
      \lesssim&~h^{m+1}\norm{v}_{m+1}+h^{m+1}\norm{v}_{m+1}\\
      \lesssim&~h^{m+1}\norm{v}_{m+1},
    \end{align*}
    which completes the proof of (\ref{QhNh_L2}) and (\ref{Nh_L2}).

    To verify the estimate (\ref{Rh_L2}), we consider the dual problem as follows
      \begin{align}
        \label{dual-equ11}-\Delta \phi =&~ Q_0v-\Pi_0^Rv,\quad in ~\Omega,\\
        \label{dual-equ22}\phi =&~ 0,\quad\qquad\qquad\,\, on ~\partial \Omega.
      \end{align}
    Using the above problem, $Q_hv-\Pi_h^Rv\in V_h^0$ and $v|_{\partial \Omega}=0$, we can obtain the same derivation process and results. The proof is completed.
  \end{proof}

  \begin{corollary}
      For $v\in H^1_0(\Omega)\bigcap H^{m+1}(\Omega)$ or $\overline H ^1(\Omega)\bigcap H^{m+1}(\Omega)$, where $1\leq m\leq k$, we have
    \begin{align}
      \label{QhRh_0h}\norm{Q_hv-\Pi_h^Rv}_{0,h}\lesssim&~h^{m+1}\norm{v}_{m+1},\\
      \label{QhNh_0h}\norm{Q_hv-\Pi_h^Nv}_{0,h}\lesssim&~h^{m+1}\norm{v}_{m+1},
    \end{align}
    Furthermore, we acquire
    \begin{align}
      \label{Rh_0h}\norm{v-\Pi_h^Rv}_{0,h}\lesssim&~h^{m+1}\norm{v}_{m+1},\\
      \label{Nh_0h}\norm{v-\Pi_h^Nv}_{0,h}\lesssim&~h^{m+1}\norm{v}_{m+1},
    \end{align}
  \end{corollary}

  \begin{proof}
    By the definition of the norm $\norm{\cdot}_{0,h}$, (\ref{QhRh_L2}), Lemma \ref{lemma-norm-equ}, (\ref{trbv_Qhv}) and (\ref{Rh_trb}), we have
    \begin{align*}
      &~\norm{Q_hv-\Pi_h^Rv}_{0,h}^2\\
      =&~ \sumT\norm{Q_0v-\Pi_0^Rv}_T^2+\sumT h\norm{Q_0v-\Pi_0^Rv-(Q_bv-\Pi_b^Rv)}_{\partial T}^2\\
      \lesssim&~ h^{2(m+1)}\norm{v}_{m+1}^2+h^2\sumT h^{-1}\norm{Q_0v-\Pi_0^Rv-(Q_bv-\Pi_b^Rv)}_{\partial T}^2\\
      \lesssim&~ h^{2(m+1)}\norm{v}_{m+1}^2+h^2\norm{Q_hv-\Pi_h^Rv}_{1,h}^2\\
      \lesssim&~ h^{2(m+1)}\norm{v}_{m+1}^2+h^2\trb {Q_hv-\Pi_h^Rv}^2\\
      \lesssim&~ h^{2(m+1)}\norm{v}_{m+1}^2+h^2(\trb {Q_hv-v}+\trb {v-\Pi_h^Rv})^2\\
      \lesssim&~h^{2(m+1)}\norm{v}_{m+1}^2+h^{2m+2}\norm{v}_{m+1}^2\\
      \lesssim&~h^{2(m+1)}\norm{v}_{m+1}^2.
    \end{align*}
    Using the definitions of $\norm{\cdot}_{0,h}$ and the projection $Q_b$, the trace inequality, and the projection inequality, we obtain
    \begin{align*}
      &~\norm{v-\Pi_h^Rv}_{0,h}^2\\
      \lesssim&~\norm{v-Q_hv}_{0,h}^2+\norm{Q_hv-\Pi_h^Rv}_{0,h}^2\\
      \lesssim&~\sumT \norm{v-Q_0v}_T^2+\sumT h\norm{v-Q_0v-(v-Q_bv)}_{\partial T}^2+h^{2(m+1)}\norm{v}_{m+1}^2\\
      \lesssim&~\sumT \norm{v-Q_0v}_T^2+\sumT h\norm{v-Q_0v}_{\partial T}^2+h^{2(m+1)}\norm{v}_{m+1}^2\\
      \lesssim&~\sumT \norm{v-Q_0v}_T^2+\sumT h^2\norm{\G (v-Q_0v)}_{ T}^2+h^{2(m+1)}\norm{v}_{m+1}^2\\
      \lesssim&~h^{2(m+1)}\norm{v}_{m+1}^2.
    \end{align*}
    We complete the proof of (\ref{QhRh_0h}) and (\ref{Rh_0h}). Analogously, we get (\ref{QhNh_0h}) and (\ref{Nh_0h}).
  \end{proof}


\begin{thebibliography}{10}

	\bibitem{MR1007396}
	{\sc D.~N. Arnold, L.~R. Scott, and M.~Vogelius}, {\em Regular inversion of the
	  divergence operator with {D}irichlet boundary conditions on a polygon}, Ann.
	  Scuola Norm. Sup. Pisa Cl. Sci. (4), 15 (1988), pp.~169--192.
	
	\bibitem{MR391538}
	{\sc F.~Brezzi}, {\em Sur la m\'{e}thode des \'{e}l\'{e}ments finis hybrides pour le
              probl\`eme biharmonique}, Numer. Math., 24 (1975), pp.~103--131.
	
	\bibitem{StokesSFWG}
	{\sc Y.~Feng, Y.~Liu, R.~Wang, and S.~Zhang}, {\em A stabilizer-free weak
	  {G}alerkin finite element method for the {S}tokes equations}, Adv. Appl.
	  Math. Mech., 14 (2022), pp.~181--201.
	
	\bibitem{gastaldi_sharp_1989}
	{\sc L.~Gastaldi and R.~H. Nochetto}, {\em Sharp maximum norm error estimates
	  for general mixed finite element approximations to second order elliptic
	  equations}, RAIRO Mod\'{e}l. Math. Anal. Num\'{e}r., 23 (1989), pp.~103--128.
	
	\bibitem{MR2817542}
	{\sc J.~Hu, Y.~Huang, and S.~Zhang}, {\em The lowest order differentiable
	  finite element on rectangular grids}, SIAM J. Numer. Anal., 49 (2011),
	  pp.~1350--1368.
	
	\bibitem{HJMFEM}
	{\sc C.~Johnson}, {\em On the convergence of a mixed finite-element method for
	  plate bending problems}, Numer. Math., 21 (1973), pp.~43--62.
	
	\bibitem{MR0386298}
	{\sc T.~Miyoshi}, {\em A finite element method for the solutions of fourth
	  order partial differential equations}, Kumamoto J. Sci. (Math.), 9 (1972/73),
	  pp.~87--116.
	
	\bibitem{BiharmonicMixFEM3}
	{\sc P.~Monk}, {\em A mixed finite element method for the biharmonic equation},
	  SIAM Journal on Numerical Analysis, 24 (1987), pp.~737--749.
	
	\bibitem{BiharmonicNCFEM}
	{\sc L.~S.~D. Morley*}, {\em The triangular equilibrium element in the solution
	  of plate bending problems}, The Aeronautical Quarterly,  (1968),
	  pp.~149--169.
	
	\bibitem{BiharmonicDGFEM}
	{\sc I.~Mozolevski, E.~S\"{u}li, and P.~R. B\"{o}sing}, {\em {$hp$}-version a
	  priori error analysis of interior penalty discontinuous {G}alerkin finite
	  element approximations to the biharmonic equation}, J. Sci. Comput., 30
	  (2007), pp.~465--491.
	
	\bibitem{BiharmonicWGMFEM}
	{\sc L.~Mu, J.~Wang, Y.~Wang, and X.~Ye}, {\em A weak {G}alerkin mixed finite
	  element method for biharmonic equations}, in Numerical solution of partial
	  differential equations: theory, algorithms, and their applications, vol.~45
	  of Springer Proc. Math. Stat., Springer, New York, 2013, pp.~247--277.
	
	\bibitem{BrinkmanWG}
	{\sc L.~Mu, J.~Wang, and X.~Ye}, {\em A stable numerical algorithm for the
	  {B}rinkman equations by weak {G}alerkin finite element methods}, J. Comput.
	  Phys., 273 (2014), pp.~327--342.
	
	\bibitem{BiharmonicWG}
	{\sc L.~Mu, J.~Wang, and X.~Ye}, {\em Weak {G}alerkin
	  finite element methods for the biharmonic equation on polytopal meshes},
	  Numer. Methods Partial Differential Equations, 30 (2014), pp.~1003--1029.
	
	\bibitem{MFEMTheory}
	{\sc J.~T. Oden}, {\em Generalized conjugate functions for mixed finite element
	  approximations of boundary value problems}, in The mathematical foundations
	  of the finite element method with applications to partial differential
	  equations ({P}roc. {S}ympos., {U}niv. {M}aryland, {B}altimore, {M}d., 1972),
	  Academic Press, New York-London, 1972, pp.~626--669.
	
	\bibitem{MR1191139}
	{\sc P.~Oswald}, {\em Hierarchical conforming finite element methods for the
	  biharmonic equation}, SIAM J. Numer. Anal., 29 (1992), pp.~1610--1625.
	
	\bibitem{CRMFEM}
	{\sc P.~C. RAVIART}, {\em A mixed finite element method for the biharmonic
	  equation}, Mathematical Aspects of Finite Elements in Partial Differential
	  Equations,  (1974), pp.~125--145.
	
	\bibitem{wang_asymptotic_1989}
	{\sc J.~Wang}, {\em Asymptotic expansions and {$L^\infty$}-error estimates for
	  mixed finite element methods for second order elliptic problems}, Numer.
	  Math., 55 (1989), pp.~401--430.
	
	\bibitem{PossionMixedWG}
	{\sc J.~{}Wang and X.~Ye}, {\em A weak galerkin mixed finite element method for
	  second order elliptic problems}, 83, pp.~2101--2126.
	
	\bibitem{PossionWG}
	{\sc J.~Wang and X.~Ye}, {\em A weak galerkin finite element method for
	  second-order elliptic problems}, Journal of Computational and Applied
	  Mathematics, 241 (2013), pp.~103--115.
	
	\bibitem{StokesWG}
	{\sc J.~Wang and X.~Ye}, {\em A weak {G}alerkin finite element method for the
	  stokes equations}, Adv. Comput. Math., 42 (2016), pp.~155--174.
	
	\bibitem{PossionSFWG}
	{\sc X.~Ye and S.~Zhang}, {\em A stabilizer-free weak {G}alerkin finite element
	  method on polytopal meshes}, J. Comput. Appl. Math., 371 (2020), pp.~112699,
	  9.
	
	\bibitem{BiharmonicSFWG}
	{\sc X.~Ye and S.~Zhang}, {\em A stabilizer free
	  weak {G}alerkin method for the biharmonic equation on polytopal meshes}, SIAM
	  J. Numer. Anal., 58 (2020), pp.~2572--2588.
	
	\bibitem{BiharmonicWGReOrder}
	{\sc R.~Zhang and Q.~Zhai}, {\em A weak {G}alerkin finite element scheme for
	  the biharmonic equations by using polynomials of reduced order}, J. Sci.
	  Comput., 64 (2015), pp.~559--585.
	
	\bibitem{BiharmonicCFEM2}
	{\sc S.~Zhang}, {\em A {$C_1$}-{$P_2$} finite element without nodal basis},
	  M2AN Math. Model. Numer. Anal., 42 (2008), pp.~175--192.
	
	\bibitem{SFC0WGBiharmonic}
	{\sc P.~Zhu, S.~Xie, and X.~Wang}, {\em A stabilizer-free {$C^0$} weak
	  {G}alerkin method for the biharmonic equations}, Sci. China Math., 66 (2023),
	  pp.~627--646.
	
\end{thebibliography}
\end{document}